\documentclass[9pt,reqno]{amsart}

\usepackage{amsmath}
\usepackage{amsthm}
\usepackage{amsfonts}
\usepackage{amssymb}
\usepackage{graphicx}
\usepackage{bm}
\usepackage{multirow}%

\usepackage{flafter}
\usepackage{bm}
\usepackage[all]{xy}
\usepackage[left,modulo]{lineno}
\usepackage{multirow}
\usepackage{lscape}

\newcommand{\E}{\mathbb{E}}

\newtheorem{theorem}{Theorem}
\newtheorem{theoremA}{Theorem }

\newtheorem{corollary}[theorem]{Corollary}

\newtheorem{definition}[theorem]{Definition}

\newtheorem{example}[theorem]{Example}

\newtheorem{lemma}[theorem]{Lemma}
\newtheorem{notation}[theorem]{Notation}
\newtheorem{proposition}[theorem]{Proposition}
\newtheorem{remark}[theorem]{Remark}
\renewenvironment{proof}{\textit{Proof}}{${}\hfill\square$}

\def\~{\tilde}

\begin{document}

\title{Projective limit of a sequence of compatible weak symplectic forms on a sequence of Banach bundles and Darboux Theorem}

\author{Fernand Pelletier}

\address{Unit\'e Mixte de Recherche 5127 CNRS, Universit\'e  de Savoie Mont Blanc, Laboratoire de Math\'ematiques (LAMA),Campus Scientifique,  73370 Le Bourget-du-Lac, France}
\email{fernand.pelletier@univ-smb.fr}

\date{}
\maketitle

\begin{abstract}
 Given a projective sequence of  Banach bundles, each one provided with a of weak symplectic form, we look for conditions   under which, the corresponding sequence of  weak symplectic forms gives rise to weak symplectic form on  the projective  limit bundle.  Then we apply this results to the tangent bundle of a projective limit of Banach manifolds.
This naturally leads  to  ask about conditions under which  the Darboux Theorem is also true on the projective limit of Banach manifolds. We will  give some necessary and some sufficient conditions so that  such a result is true.
Then  we discuss why, in general,  the Moser's method can not work on projective limit of Banach  weak symplectic Banach manifolds
without very strong conditions like   Kumar 's results (\cite{Ku2}). In particular we give an example   of a projective sequence of 
 weak symplectic Banach manifolds on which the Darboux Theorem is  true on each manifold,  but  is not true  on the projective limit of these manifolds.\\ 

\tiny{{\it 2010  MSC:}  53D35, 55P35.
}
 \end{abstract}


\section{Introduction}


In the Banach context, it is well known that a symplectic form can be strong
or weak (see Definition \ref{D_WeaklyNonDegenerateBilinearForm}). The Darboux
Theorem was firstly proved for strong  symplectic  Banach manifolds  by Weinstein (\cite{Wei}). But
Marsden (\cite{Ma2}) showed that the Darboux theorem fails for a weak
symplectic Banach manifold. However Bambusi \cite{Bam} found necessary and
sufficient conditions for the validity of Darboux theorem for a weak
symplectic Banach manifold (\textit{Darboux-Bambusi Theorem}). The proofs of all
these versions of Darboux Theorem were all established by Moser's
method.\newline
In a wider context like Fr\'{e}chet or convenient manifolds, a symplectic form
is always weak.  Recently, a new
approach to differential geometry in Fr\'{e}chet context was initiated and
developed by G. Galanis, C. T. J. Dodson, E. Vassiliou and their collaborators
in terms of projective limits of Banach manifolds (see \cite{DGV} for a
panorama of these results). In this situation, P. Kumar, in \cite{Ku2}, proves
a version of Darboux Theorem, by Moser method, for a projective sequence of weak symplectic manifolds which satisfy  the assumption of the Darboux-Bambusi Theorem
but  under very strong added conditions on this sequence. On the other hand,  a metric approach  of differential geometry 
on Fr\'echet manifold   was firstly introduced by Muller. 
This concept   gives rise to Keller-differentiable calculus as exposed in details  by Glockner  
in \cite{Glo}. In this way we can consider  the so called bounded Fr\'echet framework (cf. \cite{Mul})
 in which a  classical  implicit function Theorem is true   and a Theorem of existence of local flow can be proved (cf. \cite{Eft1}). In this context  
 Eftekharinasab  in  \cite{Eft2}, proves
a version of Darboux Theorem using  Moser's method  also  under very strong assumptions. In fact when such a Fr\'echet   manifold is also a projective limit of    Banach manifolds this result seems to recover Kumar's result.\\

More generally we can look for  conditions under which a family of weak symplectic forms on a projective sequence of Banach bundles gives rise to a weak symplectic form on the projective limit bundle:\emph{ this is the essential purpose of this paper}. \\
${}\;\;\;$ Of course this naturally  gives rise to application to projective limit of weak Banach manifold and  the problem of the existence of a Darboux theorem on the projective limit of weak symplectic Banach manifolds under the assumption of Darboux-Bambusi Theorem.\\

More precisely let $\left( \mathbb{E}_i,\ell_i^j \right) _{j\geq i}$ be a  reductive\footnote{ {\it i.e.} $\ell_i^j(\mathbb{E}_j)$ is dense in $\mathbb{E}_i$}  projective sequence of  Banach spaces and $(\omega_i) _{i \in \mathbb{N}}$ be a sequence of (linear)  weak symplectic forms $\omega_i$ on $\mathbb{E}_i$. We say that $(\omega_i) _{i \in \mathbb{N}}$ is a   sequence of compatible symplectic forms if each ${\ell_i^{i+1}}$ satisfies
\begin{equation*}\label{isometryii+1}
\ker\ell_i^{i+1} \cap (\ker\ell_i^{i+1})^\perp=\{0\}  \textrm{ and } (\ell_i^{i+1})^*\omega_{i}=\omega_{i+1}\textrm{ in restriction to}  (\ker \ell_i^{i+1})^{\perp}
\end{equation*}
where $(\ker\ell_i^{i+1})^\perp$ is the orthogonal of $\ker\ell_i^{i+1}$ relatively to $\omega_{i+1}$ (cf. Defintion \ref{D_IsometricLinear2Forms})

Now consider  a reductive  projective sequence  of Banach bundles  $\left( E_i,\lambda_i^j \right)  _{\underleftarrow{i}}$ over a projective sequence $\left( M_i,\delta_i^j \right)  _{j\geq i}$ of Banach  manifolds and let $(\omega_i)_{i \in \mathbb{N}}$ be a sequence   of weak symplectic forms $\omega_i$ on $E_i$. We say that  $(\omega_i)_{i \in \mathbb{N}}$ is a sequence of compatible symplectic forms if, for any $x_i\in M_i$ and $i\in \mathbb{N}$, the sequence $\left((\omega_i)_{x_i}\right)_ {i \in \mathbb{N}}$ is a  sequence of compatible  (linear)  weak symplectic forms on the projective sequence $\left( \pi_i^{-1}(x_i),(\lambda_i^j)_{x_j} \right) _{j\geq i}$ of Banach spaces (cf. Definition \ref{D_IsometricWeakSymplecticOnProjectiveBanachBundles})

In this  context ,  we have (cf. Theorem \ref{T_CoherentWeakSymplecticBanachManifold} and Corollary \ref{C_CharacterizationOfWeakSymplectic2FormOnProjectiveLimitOfSubmersiveSequence}):
\begin{theoremA}\label{A1} Consider  a  reductive\footnote{that is the projective sequence of typical fiber $\left( \mathbb{E}_i,\overline{\lambda_i^j}\right)$ is a reduced  projective sequence of Banach spaces} projective sequence $\left(E _i,\lambda_i^j\right)  _{\underleftarrow{i}}$ of Banach bundles over a projective sequence of Banach manifolds $\left( M_i,\delta_i^j\right)  _{j\geq i}$ .
\begin{enumerate}
\item
Let $(\omega_i)_{i \in \mathbb{N}}$ be a sequence of compatible   weak symplectic forms on    $\left(E _i,\lambda_i^j\right)  _{\underleftarrow{i}}$. Then  $\omega=\underleftarrow{\lim}\omega_i$ is a well defined  weak symplectic form on the Fr\'echet bundle $E=\underleftarrow{\lim}E_i$  over $M=\underleftarrow{\lim}M_i$.
\item
Conversely, let $\omega$ be  a weak symplectic form  on a projective limit bundles  $(E=\underleftarrow{\lim}E_i,\pi=\underleftarrow{\lim}\pi_i, M=\underleftarrow{\lim} M_i)$ of a submersive\footnote{cf. Definition \ref{D_StrongProjectiveLimitOfBanachBundle} }  projective sequence of Banach bundles$\left(E _i,\lambda_i^j\right)  _{\underleftarrow{i}}$.
  Assume that  for each $x=\underleftarrow{\lim}x_i$,    the map $(\lambda_i)_x:\pi^{-1}(x) \to \pi_i^{-1}(x_i)$ is a  symplectic  submersion   \footnote{cf Definition \ref{D_SymplecticSubmersion}}.
Then $\omega$  induces  a weak symplectic $2$-form $\omega_i$ on $E_i$ which  gives rise to a  family of compatible weak symplectic forms.  Moreover, the  $2$-form on $E$  defined by this sequence $(\omega_i)$ is precisely the given $2$-form $\omega$.
\item
A $2$-form $\omega$ on a Fr\'echet bundle, projective limit  $(E=\underleftarrow{\lim}E_i, \pi=\underleftarrow{\lim}\pi_i, M=\underleftarrow{\lim}M_i)$ of a submersive sequence of Banach fibre bundles $\left( E_i,  \lambda_i^j\right)  _{\underleftarrow{i}}$ is a weak symplectic form if and only if there exists a sequence of coherent  weak symplectic forms $(\omega_i)_{i \in \mathbb{N}}$ on $E_i$ such that $\omega=\underleftarrow{\lim}\omega_i$.\\
\end{enumerate}
\end{theoremA}

As corollary we obtain (cf. Theorem \ref{C_CoherentWeakSymplecticBanachManifold}):\\

\begin{theoremA}\label{A2}${}$
\begin{enumerate}
\item
Let $\left( M_i,\lambda_i^j\right)  _{j\geq i}$ be a reduced sequence of Banach manifolds and $(\omega_i)_{i \in \mathbb{N}}$ a sequence of compatible  weak symplectic forms. Then  $\omega=\underleftarrow{\lim}\omega_i$ is a weak symplectic form on  $M=\underleftarrow{\lim}M_i$
\item
Let $\omega$ be a $2$-form on a projective limit $M=\underleftarrow{\lim}M_i$  of a submersive sequence of manifolds $\left( M_i,\delta_i^j\right)  _{j\geq i}$. Then $\omega $ is a weak symplectic form if and only if each $T_x\delta_i: T_xM\to T_{x_i}M_i$ is  symplectic   submersion for each $i\in \mathbb{N}$.\\
\end{enumerate}
\end{theoremA}

 Now, in the context of Theorem \ref{A2} Point (1),   assume that each weak symplectic manifold $(M_i ,\omega_i)$ \emph{satisfies  the assumptions of Darboux-Bambusi Theorem (cf. Theorem \ref{T_localDarboux})} for each $i\in \mathbb{N}$, then it follows that  \emph{ the same is true for the projective limit $(M=\underleftarrow{\lim}M_i, \omega=\underleftarrow{\lim}\omega_i)$} (cf. Theorem \ref{T_omega_nCoherentM}).\\
${}\;$ Since the Darboux-Bambusi Theorem is then true for each $(M_i ,\omega_i)$ it seems natural  to look for the same result for $(M,\omega)$. A partial answer is given in Theorem \ref{T_equivDarbouxProjective}.\\

 The last section  of this paper is devoted to a discussion on how the Moser's method in the previous context can be applied. In particular Theorem \ref{T_Uniformly BoundedSymplectic} gives (very strong ) sufficient conditions under which the Moser's method can be applied and which is a kind of generalization of Kumar's result (\cite{Ku1}).\\
${}\;$ \emph{ Unfortunately such kind of results require so strong assumptions, that it seems there is no concrete applications outside  elementary examples. }\\
 Finally, we give some examples for which the Darboux-Bambusi  theorem is true and an example for which the Darboux-Bambusi Theorem is  true on each manifold,  but  is not true  on the projective limit of these manifolds. Note that this last section  is analog to the same type of discussion  in \cite{Pe} in the context of  direct limit of weak symplectic manifolds.\\

This work is self contained.

\noindent In section 2,  after a survey on  known results on symplectic forms on a Banach space ($\S$2.1)
we look for properties of  a  sequence of  compatible (linear) symplectic forms on a projective sequences of Banach spaces ($\S$ 2.2). The precise context of 
Theorem \ref{A1} (resp. Theorem \ref{A2}) can be found in $\S$ 2.3 (resp. $\S$2.4). The proofs of all these results take place in $\S$2.5.\\
Section 3 is devoted to  show that under  the assumption of Theorem \ref{A2}, the  Darboux-Bambusi assumptions which are satisfied for a projective limit of weak symplectic manifold are also valid on its projective limit. The first subsection recall the Moser's method and the Darboux-Bambusi Theorem. In the next subsection,  under assumption of Theorem \ref{A1} (1),   for  such a projective limit of Banach bundles which satisfy a generalization of Darboux-Bambusi Theorem  assumptions, we show that  its  projective limit has the same properties. The last section is a discussion on the  problem of existence of Darboux charts on a strong reduced projective sequence of Banach manifolds. Sufficient conditions are given in the first subsection.The announced discussion is developed in $\S$4.2. Examples and contre-example about the existence of a projective limit of Darboux charts are given in $\S$4.3.
Finally we end this paper by a series of Appendices which sumrize  all the definitions and properties on projective limits needed in this paper.\\


\section{Projective limit of a coherent  sequence of  weak symplectic forms on a projective sequence of Banach bundle}
\label{__ProjectiveLimitOfProjectiveSequenceOfWeakSymplectiqueBanachBundle}

\subsection{Symplectic forms on Banach space}

\label{linearsymplectic}

In this section we  recall some well known results on linear symplectic forms on a Banach space (cf. for instance \cite{Pe}):

\begin{definition}
\label{D_WeaklyNonDegenerateBilinearForm}Let $\mathbb{E}$ be a Banach space. A
bilinear form $\omega$ is said to be weakly non degenerate if
$\left(  \forall Y\in\mathbb{E},\ \omega\left(  X,Y\right)  =0\right)
\quad\Longrightarrow\quad X=0$.

\end{definition}

Classically, to $\omega$ is associated the linear map 

$\;\;\;\; \omega^{\flat}:  \mathbb{E}  \longrightarrow  \mathbb{E}^{\ast}$
\noindent defined by $ \left(\omega^{\flat}(X)\right)(Y)=\omega\left(  X,Y\right),\;:\;  \forall Y\in \mathbb{E}$.

Clearly,  $\omega$ is weakly non degenerate if and only if  $\omega^{\flat}$ is
injective. \\
 The $2$-form $\omega$ is  called  \textit{strongly nondegenerate} if $\omega^\flat$ is an isomorphism.


\bigskip

A fundamental result in finite dimensional linear symplectic space is the
following:

If $\omega$ is a symplectic form on a finite dimensional vector space
$\mathbb{E}$, there exists a vector space $\mathbb{L}$ and an isomorphism
$A:\mathbb{E}\rightarrow\mathbb{L}\oplus\mathbb{L}^{\ast}$ such that
$\omega=A^{\ast}\omega_{_{\mathbb{L}}}$ where
\begin{equation}
\omega_{_{\mathbb{L}}}((u,\eta),(v,\xi)=<\eta,v>-<\xi,u> \label{Darboux}%
\end{equation}
This result is in direct relation with the notion of \textit{Lagangian
subspace} which is a fundamental tool in the finite dimensional symplectic framework.

\bigskip

In the Banach framework, let $\omega$ be a weak symplectic form on a Banach space.

A subspace $\mathbb{F}$ is \textit{isotropic} if $\omega(u,v)=0$ for all
$u,v\in\mathbb{F}$. An isotropic subspace is always closed. \newline If
$\mathbb{F}^{\perp_{\omega}}=\{w\in\mathbb{E}\;:\forall u\in\mathbb{F}%
,\ \omega(u,v)=0\;\}$ is the orthogonal symplectic space of $\mathbb{F}$, then
$\mathbb{F}$ is isotropic if and only if $\mathbb{F}\subset\mathbb{F}%
^{\perp_{\omega}}$ and is \textit{maximal isotropic if }$\mathbb{F}%
=\mathbb{F}^{\perp_{\omega}}$\textit{.} Unfortunately, in the Banach
framework, a maximal isotropic subspace $\mathbb{L}$ can be not supplemented.
Following Weinstein's terminology (\cite{Wei}), an isotropic space
$\mathbb{L}$ is called a \textit{Lagrangian space} if there exists an
isotropic space $\mathbb{L}^{\prime}$ such that $\mathbb{E}=\mathbb{L}%
\oplus\mathbb{L}^{\prime}$. Since $\omega$ is strong non degenerate, this
implies that $\mathbb{L}$ and $\mathbb{L}^{\prime}$ are maximal isotropic and
then are Lagrangian spaces (see \cite{Wei}).

Unfortunately, in general, for a given symplectic structure, Lagrangian
subspaces need not exist (cf. \cite{KaSw}). Even for a strong
symplectic structure on Banach space which is not Hilbertizable, the non
existence of Lagrangian subspaces is an open problem to our knowledge.
Following \cite{Wei}, a symplectic form $\omega$ on a Banach space
$\mathbb{E}$ is a \textit{Darboux (linear) form} if there exists a Banach
space $\mathbb{L}$ and an isomorphism $A:\mathbb{E}\rightarrow\mathbb{L}%
\oplus\mathbb{L}^{\ast}$ such that $\omega=A^{\ast}\omega_{\mathbb{L}}$ where
$\omega_{\mathbb{L}}$ is defined in (\ref{Darboux}). Note that in this case
$\mathbb{E}$\textit{ must be reflexive}.

\bigskip

Let $\mathbb{E}$ be a Banach space provided with a norm $||\;||$. We consider
a symplectic form $\omega$ on $\mathbb{E}$ and let $\omega^{\flat}%
:\mathbb{E}\rightarrow\mathbb{E}^{\ast}$ be the associated bounded linear
operator. Following \cite{Bam} and \cite{Ku1}, on $\mathbb{E}$, we consider
the norm $||u||_{\omega}=||\omega^{\flat}(u)||^{\ast}$ where $||\;||^{\ast}$
is the canonical norm on $\mathbb{E}^{\ast}$ associated to $||\;||$. Of
course, we have $||u||_{\omega}\leq||\omega^{\flat}||^{\mathrm{op}}.||u||$
(where $||\omega^{\flat}||^{\mathrm{op}}$ is the norm of the operator
$\omega^{\flat}$) and so the inclusion of the normed space $(\mathbb{E}%
,||\;||)$ in $(\mathbb{E},||\;||_{\omega})$ is continuous. We denote by
$\widehat{\mathbb{E}}$ the Banach space which is the completion of
$(\mathbb{E},||\;||_{\omega})$. Since $\omega^{\flat}$ is an isometry from
$(\mathbb{E},||\;||_{\omega})$ to its range in $\mathbb{E}^{\ast}$, we can extend
$\omega^{\flat}$ to a bounded operator $\widehat{\omega}^{\flat}$ from
$\widehat{\mathbb{E}}$ to $\mathbb{E}^{\ast}$. Assume that $\mathbb{E}$ is
\textit{reflexive}. Therefore $\hat{\omega}^{\flat}$ is an isometry between
$\widehat{\mathbb{E}}$ and $\mathbb{E}^{\ast}$ (\cite{Bam} Lemma 2.7).
Moreover, $\omega^{\flat}$ can be seen as a bounded linear operator from
$\mathbb{E}$ to $\widehat{\mathbb{E}}^{\ast}$ and is in fact an isomorphism
(\cite{Bam} Lemma 2.8). Note that since $\widehat{\mathbb{E}}^{\ast}$
\textit{is reflexive}, this implies that $\hat{\mathbb{E}}$ \textit{is also
reflexive.} 

\begin{remark}\label{indnorm} If ${||\;||'}$ is an equivalent norm of $||\;||$  on $\E$, then the corresponding $(||\;||')^*$ and $||\;||^*$ are also equivalent norm on $\E^*$ and so ${||\;||'}_\omega$ and $||\;||_\omega$ are equivalent norms on $\E$  and so the completion $\widehat{\E}$   depends only of Banach structure on $\E$  defined by equivalent the norms on $\E$
\end{remark}

\subsection{Case of projective limit of  Banach spaces}\label{___CaseOf Banachspaces}
Let $\omega$ be  a skew-symmetric  bilinear form on a Banach space $\mathbb{E}$ and $\mathbb{K}$ a Banach subspace of $\mathbb{E}$. Recall that the \emph{$\omega$-orthogonal subspace }$\mathbb{K}^{\perp_\omega}$ is defined by
\[
\mathbb{K}^{\perp_\omega}=\{x\in \mathbb{E},\;:\;\forall y\in \mathbb{K}, \, \omega(x,y)=0 \}.
\]
When there is no ambiguity this set is simply denoted  $\mathbb{K}^{{\perp}}$. Note that since $\omega$ is  skew-symmetric  $\mathbb{K}^{\perp}=\{x\in \mathbb{E}:\;\forall y\in \mathbb{K}, \, \omega(y,x)=0 \}$ and so $(\mathbb{K}^\perp)^\perp=\mathbb{K}$.\\
If $\mathbb{K}^0=\{\xi\in \mathbb{E}^*:\;\forall u\in \mathbb{K},\; \xi(u)=0\}$ is the annihilator\index{annihilator} of $\mathbb{K}$, then $\mathbb{K}^\perp=(\omega^\flat)^{-1}(\mathbb{K}^0)$.\\
Given  two  Banach subspaces  $\mathbb{K}$ and $ \mathbb{K}'$  of $\mathbb{E}$, the following relations are classical:
\begin{itemize}
\item[--]
If  $\mathbb{K}\subset \mathbb{K}'$ then  ${\mathbb{K}'}^{{\perp}}\subset \mathbb{K}^{{\perp}}$ and, in particular, for any subspace $\mathbb{K}$, $\mathbb{E}^{{\perp}}\subset \mathbb{K}^{{\perp}}$ .
\item[--]
$(\mathbb{K}+\mathbb{K}' )^\perp=\mathbb{K}^\perp\cap {\mathbb{K}' }^\perp$.
\item[--]
$(\mathbb{K}\cap\mathbb{K}' )^\perp=\mathbb{K}^\perp+{\mathbb{K}' }^\perp$. 	
\end{itemize}

Let $\omega$  (resp. $\omega'$) be  a skew-symmetric bilinear form on a Banach space $\mathbb{E}$ (resp. $\mathbb{E}'$) and  $\ell: \mathbb{E}\to \mathbb{E}'$ a continuous map. By analogy with the terminology for Finsler geometry  (cf. \cite{A-PD})  we introduce

\begin{definition}
\label{D_IsometryBetweenBilinearForms}  We  say that  $\ell$  is a weak    isometry  between $\omega$ and $\omega'$ if $\ell(E)$ is dense in $\mathbb{E}'$ and we have:
\begin{eqnarray}
\ker\ell \cap (\ker\ell)^\perp=\{0\}  \textrm{ and } \ell^*\omega'={\omega} \textrm{ in restriction to }  (\ker \ell)^{\perp}
\end{eqnarray}
\end{definition}
Note that the condition "$\ker\ell \cap (\ker\ell)^\perp=\{0\} $"  is equivalent to the condition "the restriction of $\omega$ to $(\ker\ell)^\perp$ is non degenerate".
\begin{proposition}
\label{P_IsometricProperties}
Let $\omega$  (resp. $\omega'$) be   a  weak skew symmetric form on a Banach space $\mathbb{E}$ (resp. $\mathbb{E}'$) and let $\ell: \mathbb{E}\to \mathbb{E}'$  be a continuous map. We set  $\mathbb{K}=(\ker \ell)$ and denote by  $\overline{\omega}$ the restriction of $\omega$ to   $\mathbb{K}^\perp$. We have the following properties:
\begin{enumerate}
\item
If $\ell $ is a weak  isometry between $\omega$ and $\omega'$, then $\overline{\omega}$ and $\omega'_{| \ell(\mathbb{E})}$ are  non degenerate, and $\ker \ell^*\omega'=\ker\ell$.
\item
If $\omega$ and $\overline{{\omega}}$ are  non degenerate,  $\ell$ is a weak   isometry  between $\omega$ and $\omega'$  if and only if $\ell^*\omega'={\omega}$ on ${ (\ker \ell)^{\perp}}$ and, in this case,  the restriction of $\omega'$ to $\ell(\mathbb{E})$ is  non degenerate.
\item
Let $\omega''$ be a skew symmetric bilinear form on a Banach space $\mathbb{E}''$ and $\ell':\mathbb{E}'\to\mathbb{E}''$ a continuous linear map. If $\ell$ (resp. $\ell'$) is a weak   isometry between $\omega$ and $\omega'$ (resp. $\omega'$ and $\omega''$)   then   $\ell'\circ \ell$ is a weak   isometry  between $\omega$ and $\omega''$.\\
\end{enumerate}
\end{proposition}

Note that if $\ell$ is a weak symplectic isometry between $\omega$ and $\omega'$, the restriction $\overline{\ell}$ of $\ell$ to $(\ker\ell)^\perp$ is an isomorphism and
\begin{eqnarray}
\label{eq_RestrictionKerlPerp}
\overline{\ell}^*\omega'=\ell^*\omega'_{| (\ker\ell)^\perp}=\omega_{| (\ker\ell)^\perp}
\end{eqnarray}

\begin{proof}
(1) We have  $\overline{\omega}(u, v)=0$ for all $v\in \mathbb{K}^\perp$ if and only if  $u$ belongs $\mathbb{K}\cap \mathbb{K}^\perp$ which implies that $\overline{\omega}$ is non degenerate. Since $\overline{\ell }$ is an isomorphism, from (\ref{eq_RestrictionKerlPerp}), it follows that $\omega'_{| \ell(\mathbb{E})}$ is  non degenerate.
Now, $v$ belongs to $\ker (\ell^*(\omega'))^\flat$ if and only if $\omega'(\ell(u), \ell(v))=0,\; \forall v$,  if and only if $\ell(u)=0$.\\

(2) Assume $\omega$ and  $\overline{\omega}$ are non degenerate. We must show that $\mathbb{E}=\mathbb{K}\oplus \mathbb{K}^\perp$. But  $(\mathbb{K}\oplus \mathbb{K}^\perp)^\perp=\mathbb{K}^\perp\cap \mathbb{K}=\{0\}$ and since $\omega$ is  non degenerate, it follows that $\{0\}^\perp=\mathbb{E}=\mathbb{K}\oplus \mathbb{K}^\perp$, which ends the proof of Point (2) according to relation  (\ref{eq_RestrictionKerlPerp}).\\

(3) Under the assumptions of Point (3), we have $\mathbb{E}=\ker\ell\oplus (\ker\ell)^\perp$ and $\mathbb{E}'=\ker\ell'\oplus (\ker\ell')^\perp$. Now, since the inclusion of $\ell(E)$ in $\mathbb{E}'$ is   continuous, if   $\mathbb{K}':=\ker\ell'\cap\ell(\mathbb{E})$, then $(\mathbb{K}')^\perp=(\ker\ell')^\perp\cap \ell(\mathbb{E})$ and so  $\ell(E)=\mathbb{K}'\oplus (\mathbb{K}')^\perp$. Let $\overline{\ell}$
be the restriction of $\ell$ to $(\ker\ell)^\perp$; it is is an isomorphism onto $(\mathbb{K}')^\perp$. If $\overline{\mathbb{K}}=\overline{\ell}^{-1}(\mathbb{K}')$ and $\overline{\mathbb{H}}=\overline{\ell}^{-1}((\mathbb{K}')^\perp)$
then $(\ker\ell)^\perp=\overline{\mathbb{K}}\oplus \overline{\mathbb{H}}$. By construction, $\ker\ell'\circ \ell=\ker\ell\oplus \overline{\mathbb{K}}$ and we have $(\ker\ell\oplus\overline{ \mathbb{K}})^\perp=(\ker\ell)^\perp\cap \overline{\mathbb{K}}^\perp=\overline{\mathbb{H}}$.
Indeed $\overline{\mathbb{H}}$ is contained in $(\ker\ell)^\perp$,  $\overline{\mathbb{H}}=\overline{\ell}^{-1}((\mathbb{K}')^\perp)$  and $(\overline{\ell}^{-1})^*(\omega_{| (\ker\ell)^\perp})=\omega'_{|\ell(\mathbb{E})}$, this
 implies that $\overline{\mathbb{H}}$ is the orthogonal of $\overline{\mathbb{K}}$  in $(\ker\ell)^\perp$. Now, the restriction   $\overline{\ell'\circ \ell}$ of ${\ell'\circ \ell}$ to
 $(\ker(\ell'\circ \ell))^\perp=\overline{\mathbb{H}}$
   is an isomorphism. Consider  for any $(u,v) \in \overline{\mathbb{H}}^2$ we have:
\[
\omega(u,v)=\omega'(\overline{\ell}(u),\overline{\ell}(u))=\omega''(\overline{\ell'\circ \ell}(u),\overline{\ell'\circ \ell}(v)).
\]
So the proof is completed.\\
\end{proof}


 \begin{definition}
\label{D_IsometricLinear2Forms}
Let $\left( \mathbb{E}_i,\ell_i^j \right) _{j\geq i}$ be a  reductive\footnote{cf. Definition \ref{D_ReducedProjectiveSequence}} projective sequence Banach spaces and $(\omega_i) _{i \in \mathbb{N}}$ be a sequence of (linear)  weak symplectic forms $\omega_i$ on $\mathbb{E}_i$. We say that $(\omega_i) _{i \in \mathbb{N}}$ is a   sequence of compatible symplectic forms if each ${\ell_i^{i+1}}$ is a weak  isometry between $\omega_{i+1}$ and $\omega_i$, for all $ i\in \mathbb{N}$
\end{definition}

We then have the following property:
\begin{proposition}
\label{P_PropertiesWeakIsometricLinear2Forms}
Let $(\omega_i)_{i \in \mathbb{N}}$ be a   sequence of compatible  symplectic forms on a reduced  projective sequence $\left( \mathbb{E}_i,{\ell_i^j} \right) _{j \geq i}$. 
Then if  $u=\underleftarrow{\lim}u_i$ and $v=\underleftarrow{\lim}v_i$ in $\mathbb{E}=\underleftarrow{\lim}\mathbb{E}_i$,
\[
\omega(u, v) =\underleftarrow{\lim}\omega_i(u_i, v_i)
\]
defines a weak symplectic $2$-form on $\mathbb{E}$.
\end{proposition}

Since the proof of this Proposition is very technical, the reader find it in Appendix \ref{__ProofProposition12}

\begin{remark}
\label{R_IsometricAllij}
${}$
\begin{enumerate}
\item
From the properties of the sequence $\left( {\ell_i^j} \right) _{j\geq i}$  and Proposition \ref{P_IsometricProperties}
 (3), if $(\omega_i)_{i\in \mathbb{N}}$  is a  sequence of compatible  weak symplectic $2$-forms,   then ${\ell_i^j}$ is  a weak  isometry between $\omega_{j}$ and $\omega_i$, for all $i\in \mathbb{N}$ and all $j\geq i$.
\item Consider  the assumptions of Proposition \ref{P_PropertiesWeakIsometricLinear2Forms}. If $\left( \mathbb{E}_i,{\ell_i^j} \right) _{j \geq i}$ is a $\mathsf{ILB}$ sequence (cf. Appendix B),  we have   $\ker\ell_i^j=\{0\}$ for all $j\geq i$ and  $i\in \mathbb{N}$.  Thus  $(\omega_i)_{i\in \mathbb{N}}$ is a sequence of compatible symplectic forms on this projective system if and only if, for all $j\geq i$ and  $i\in \mathbb{N}$, then $\omega_j=(\lambda_i^j)^*\omega_i$,  and $\omega_0$ is symplectic. But  in general,  if  for some pair $(i,j)$,  $\ker\ell_i^j\not=\{0\}$, the condition $\omega_j=(\ell_i^j)^*\omega_i$ implies that $\mathbb{E}_j^\perp\not=\{0\}$ and so $\omega_j$ cannot be symplectic.
\item In  Proposition \ref{P_PropertiesWeakIsometricLinear2Forms}, when
$\left( \mathbb{E}_i,\ell_i^j\right)  _{j\geq i}$ is a surjective \footnote{ that is each $\ell_i^j$ is surjective} projective sequence, the symplectic form $\omega$ on $\mathbb{E}$ has the property that the induced form on $\ker\ell_i$ is symplectic and so we
have $\mathbb{E}=\ker\ell_i\oplus (\ker\ell_i)^\perp$ where $(\ker\ell_i)^\perp$ is the orthogonal of $\ker\ell_i$ (relative to $\omega$).
\end{enumerate}
\end{remark}

 As in finite dimension, we introduce:

\begin{definition}
\label{D_SymplecticSubmersion}
Let $\mathbb{E}=\underleftarrow{\lim}\mathbb{E}_i$ a projective limit of a surjective projective sequence $\left( \mathbb{E}_i,\ell_i^j\right)  _{j\geq i}$.
Consider a (weak) symplectic form  $\omega$ on $\mathbb{E}$ such that $\mathbb{E}=\ker\ell_i\oplus (\ker\ell_i)^\perp$.  We will say that $\ell_i$ is a  symplectic submersion. \end{definition}

\begin{remark} In the context of Definition \ref{D_SymplecticSubmersion}, the restriction of $\ell_i$ to $(\ker\ell_i)^\perp$ is an isomorphism onto $\mathbb{E}_i$ and so we have a well symplectic form $\omega_i$ on $\mathbb{E}_i$ such that $\omega=\ell_i^*\omega_i$ in restriction to $(\ker\ell_i)^\perp$. Thus this definition is analog to the notion of isometric submersion between Finsler manifolds in finite dimension introduced in  \cite{A-PD} \end{remark}

We have the following type of converse of Proposition \ref{P_PropertiesWeakIsometricLinear2Forms} :

\begin{proposition}
\label{P_FamilySymplecticSubmersion}
Let $\left( \mathbb{E}_i,\ell_i^j\right)  _{j\geq i}$ be a  surjective projective sequence of Banach space and $E=\underleftarrow{\lim}E_i$. If $\omega$ is a symplectic form on $E$ such that $\ell_i: E\to E_i$ is  a symplectic   submersion for all $i\in \mathbb{N}$, then $\omega$ induces a symplectic form $\omega_i$ on $E_i$. Moreover, $(\omega_i)_{i\in \mathbb{N}}$ is a   sequence of compatible symplectic forms and the projective limit associated to this sequence is precisely $\omega$.
\end{proposition}

\begin{proof}
Since for $j\geq i$,  $\ell_i=\ell_i^j\circ \ell_j$  this implies  $\ker\ell_i=\ker\ell_j\oplus (\ell')_j^{-1}(\ker\ell_i^j)$. Thus we have  $\ker\ell_j\subset \ker \ell_i$ and so $(\ker\ell_i)^\perp\subset (\ker\ell_j)^\perp$. As we have seen previously, there exists a (unique) symplectic form $\omega_j$ on $\mathbb{E}_j$ such that  $\omega=\ell_j^*\omega_j$ on $(\ker\ell_j)^\perp$. Since for any $j\in \mathbb{N}$, the restriction  $\ell'_j$ to $(\ker\ell_j)^\perp$ is an isomorphism onto $\mathbb{E}_j$, we have:

$\omega_j(u_j, v_j)=\omega(\ell'_j(u'_j),\ell'_j(v'_j))$ for all $u'_j, v'_j\in (\ker\ell_j)^\perp$ with $u_j=\ell'_j(u'_j)$ and $v_j=\ell'_j(v'_j)$.\\

But since  $(\ker\ell_i)^\perp\subset (\ker\ell_j)^\perp$, it follows that, for any $u'_i, v'_i\in (\ker\ell_i)^\perp$, we have
\begin{align*}
\omega(\ell'_i(u'_i), \ell'_i(v'_i))   & =\omega(\ell_i^j\circ\ell'_j(u'_i),\ell_i^j\circ\ell'_j(v'_i))\\                                       & =(\ell_i^j)^*\omega(\ell'_j(u'_i),\ell'_j(v'_i).
\end{align*}
Thus we obtain
\[
(\omega_j)=(\ell_i^j)^*\omega_i \textrm { on } \ell'_j\left((\ker\ell_i)^\perp\right).
\]
The proof will be completed if we show that $\ell'_j\left((\ker\ell_i)^\perp\right)=(\ker\ell_i^j)^\perp$. But this results follows from $\ell'_j(\ker\ell_i)=\ker\ell_i^j$.\\
\end{proof}

\subsection{Case of  projective sequence of Banach bundles}
\label{__ProjectiveLimitOfProjectiveSequenceOfWeakSymplectiqueBanachBundles}

\begin{definition}
\label{D_IsometricWeakSymplecticOnProjectiveBanachBundles}
 Let $\left( E_i,\lambda_i^j \right)  _{\underleftarrow{i}}$ be a  projective sequence  of Banach bundles over a projective sequence $\left( M_i,\delta_i^j \right)  _{j\geq i}$ of manifolds and let $(\omega_i)_{i \in \mathbb{N}}$ be a sequence   of weak symplectic forms $\omega_i$ on $E_i$.
If $\mathbb{E}_i$ is the typical fibre of $E_i$, assume  that the following properties are satisfied:
\begin{description}
\item[(\textbf{RPSBS})]
The sequence $\left( \mathbb{E}_i,\overline{\lambda_i^j}\right)  _{j\geq i}$ is a reduced  projective sequence of Banach spaces.
\end{description}
We say that  $(\omega_i)_{i \in \mathbb{N}}$ is a  sequence of  compatible symplectic forms if the sequence $\left((\omega_i)_{x_i}\right)_ {i \in \mathbb{N}}$ is a  sequence of compatible (linear)  weak symplectic forms on the projective sequence $\left( \pi_i^{-1}(x_i),(\lambda_i^j)_{x_j} \right) _{j\geq i}$ of Banach spaces.
\end{definition}
Under the context of this Definition,  we have
\begin{theorem}
\label{T_CoherentWeakSymplecticBanachManifold}Consider  a  projective sequence $\left(E _i,\lambda_i^j\right)  _{\underleftarrow{i}}$ of Banach bundles over a projective sequence of Banach manifolds $\left( M_i,\delta_i^j\right)  _{j\geq i}$ which satisfies the assumption \emph{\textbf{(RPSBS)}}.
\begin{enumerate}
\item
Let $(\omega_i)_{i \in \mathbb{N}}$ be a sequence of compatible   weak symplectic forms on    $\left(E _i,\lambda_i^j\right)  _{\underleftarrow{i}}$. Then  $\omega=\underleftarrow{\lim}\omega_i$ is a well defined  weak symplectic form on the Fr\'echet bundle $E=\underleftarrow{\lim}E_i$  over $M=\underleftarrow{\lim}M_i$.
\item
Conversely, let $\omega$ be  a weak symplectic form  on a projective limit bundles  $(E=\underleftarrow{\lim}E_i,\pi=\underleftarrow{\lim}\pi_i, M=\underleftarrow{\lim} M_i)$ of a submersive\footnote{cf. Definition \ref{D_StrongProjectiveLimitOfBanachBundle} }  projective sequence of Banach bundles $\left( E_i,\pi_i, M_i \right) _{\underleftarrow{i}}$.
  Assume that  for each $x=\underleftarrow{\lim}x_i$,    the map $(\lambda_i)_x:\pi^{-1}(x) \to \pi_i^{-1}(x_i)$ is a submersion.
Then $\omega$  induces  a weak symplectic $2$-form $\omega_i$ on $E_i$ which  gives rise to a  family of compatible  weak symplectic forms.  Moreover, the  $2$-form on $E$  defined by this sequence $(\omega_i)_{i\in \mathbb{N}}$ is precisely the given $2$-form $\omega$.\\
\end{enumerate}
\end{theorem}
We obtain directly the following Corollary:
\begin{corollary}
\label{C_CharacterizationOfWeakSymplectic2FormOnProjectiveLimitOfSubmersiveSequence}
A $2$-form $\omega$ on a Fr\'echet bundle, projective limit  $(E=\underleftarrow{\lim}E_i, \pi=\underleftarrow{\lim}\pi_i, M=\underleftarrow{\lim}M_i)$ of a submersive sequence of Banach fibre bundles $\left( E_i,  \lambda_i^j\right)  _{\underleftarrow{i}}$ is a weak symplectic form if and only if there exists a sequence of compatible  weak symplectic forms $(\omega_i)_{i \in \mathbb{N}}$ on $E_i$ such that $\omega=\underleftarrow{\lim}\omega_i$.\\
\end{corollary}

Note that Theorem \ref{A1}  in the introduction  is  Theorem \ref{T_CoherentWeakSymplecticBanachManifold} joined  with Corollary \ref{C_CharacterizationOfWeakSymplectic2FormOnProjectiveLimitOfSubmersiveSequence}.\\

\subsection{Case of projective limit of weak symplectic Banach manifolds}\label{___CaseWeakSymplecticBanachManifolds}

By application of Theorem \ref{T_CoherentWeakSymplecticBanachManifold} when  $E_i$ is the tangent bundle  $TM_i$ of a Banach manifold $M_i$, we obtain the following Theorem  which is  exactly Theorem \ref{A2} in the introduction:\\

\begin{theorem}
\label{C_CoherentWeakSymplecticBanachManifold}
${}$
\begin{enumerate}
\item
Let $\left( M_i,\lambda_i^j\right)  _{j\geq i}$ be a reduced sequence of Banach manifolds and $(\omega_i)_{i \in \mathbb{N}}$ a sequence of compatible  weak symplectic forms. Then  $\omega=\underleftarrow{\lim}\omega_i$ is a weak symplectic form on  $M=\underleftarrow{\lim}M_i$
\item
Let $\omega$ be a $2$-form on a projective limit $M=\underleftarrow{\lim}M_i$  of a submersive sequence of manifolds $\left( M_i,\delta_i^j\right)  _{j\geq i}$. Then $\omega $ is a weak symplectic form if and only if each $T_x\delta_i: T_xM\to T_{x_i}M_i$ is a  symplectic  submersion  for each $i\in \mathbb{N}$.\\
\end{enumerate}
\end{theorem}

\begin{remark}
\label{R_CanonicalProjectionSymplecticSubmersion}${}$
\begin{enumerate}
 \item Given a submersive sequence $\left( M_i,\delta_i^j\right)  _{j\geq i}$ of manifolds  and a weak symplectic form $ \omega_i$ on each $M_i$,  each map $\delta_i^j:M_j\to M_i$  is a symplectic  submersion   if  the restriction of $\omega_j$ is a symplectic form on each fibre of $\delta_i^j$  and on the orthogonal symplectic of the vertical bundle of $\delta_i^j$ we have $\omega_j=(\delta_i^j)^*\omega_i$.
\item
 Let  $M=\underleftarrow{\lim}M_i$ be a  projective limit of a submersive sequence $\left( M_i,\delta_i^j\right)  _{j\geq i}$ of manifolds and $\omega$ a weak symplectic form on $M$. We say that the canonical projection $\delta_i:M\to M_i$ is a symplectic   submersion if  the restriction of $\omega$ to each fibre $\delta^{-1}(x_i)$ is a symplectic form and $\delta_i^*\omega_i=\omega$ on the orthogonal bundle of the vertical bundle of $\delta_i$.
\end{enumerate}
\end{remark}

\subsection{Proofs of results}${}$\\

\begin{proof}[Proof of Theorem \ref{T_CoherentWeakSymplecticBanachManifold}]
From Proposition \ref{P_PropertiesWeakIsometricLinear2Forms}, we know that $\omega $ is well defined.  Now $\omega $ is a smooth $2$-form since it is a projective limit of smooth $2$ forms, which ends the proof of (1).\\
Now let  $\omega$ be  a weak symplectic form on a projective limit bundle $E=\underleftarrow{\lim}E_i,\pi=\underleftarrow{\lim}\pi_i, M=\underleftarrow{\lim} M_i$ which satisfies the assumptions of (2). Given some  $x=\underleftarrow{\lim}x_i\in M$, since $(\lambda_i)_x:\pi^{-1}(x)\to \pi_i^{-1}(x_i)$ is a symplectic  submersion of symplectic spaces, the restriction $(\lambda_i)' $ of $\lambda_i$ to $\ker(\lambda_i)_x^\perp$ is an isomorphism on $\pi_i^{-1}(x_i)$ and so $(\omega_i)_{x_i}=\{[(\lambda_i)_x]^{-1}\}^*(\omega_x)_{| \ker(\lambda_i)_x^\perp}$ is a symplectic form on $\pi_i^{-1}(x_i)$. It remains to show that $x_i\mapsto \omega_{x_i}$ is smooth.\\
Fix some $x=\underleftarrow{\lim}x_i\in M$.
There exists $\phi(U) \times \mathbb{E}$ with the following commutative diagram
\begin{eqnarray}
\label{eq_diagram3Di}
\xymatrix {
    \pi^{-1}(U) \ar[rr]^{\tau} \ar[dd] \ar[dr]^{\lambda_i}  && \phi(U)\times\mathbb{E} \ar[dr]^{\overline{\delta_i}\times \overline{\lambda_i}} \ar[dd]
    \\
    & \pi_i^{-1}(U_i) \ar[rr]^{\tau_i} \ar[dd]              && \phi_i(U_i)\times \mathbb{E}_i \ar[dd] \\
   U \ar[rr]^{\;\;\;\phi} 
   \ar[dr]^{{\delta}_i}   && \phi(U) \ar[rd]^{{\overline{\delta_i}}}\\
   & U_i \ar[rr]^{\phi_i} && \phi_i(U_i) \\
 }
\end{eqnarray}
Let $\Omega$ be the symplectic form on $\phi(U)\times \mathbb{E}$ such that $\omega=\tau^*\Omega$. According to Proposition \ref{P_StrongProjectiveLimitOfBanachBundle},  $\ker\lambda_i$ is a sub-bundle of $E$. Now, since $\omega$ is a smooth symplectic form and the orthogonal  $\ker(\lambda_i)_z^\perp$ is a supplemented space of $\ker(\lambda_i)_z$ for all $z\in M$, it follows that   $\ker(\lambda_i)_z^\perp$  is a Banach sub-bunlde of  $E$ and so the Diagram (\ref{eq_diagram3Di}) have the more precise version, after shrinking $U$ if necessary:
\begin{eqnarray}
\label{eq_diagram3Disymplectic}
\xymatrix {
    \pi^{-1}(U) \ar[rr]^{\tau} \ar[dd] \ar[dr]^{\lambda_i}&& \phi(U)\times\mathbb{K}_i\times\mathbb{H}_i \ar[dr]^{\overline{\delta_i}\times \overline{\lambda_i}}\ar[dd] |!{[dl];[dr]}\hole \\
    & \pi_i^{-1}(U_i) \ar[rr]^{\tau_i} \ar[dd] && \phi_i(U_i)\times \mathbb{E}_i \ar[dd] \\
   U \ar[rr]^\phi |!{[ur];[dr]}\hole \ar[dr]^{{\delta}_i}&& \phi(U) \ar[rd]^{{\overline{\delta_i}}}\\
   & U_i \ar[rr]^{\phi_i} && \phi_i(U_i) \\
 }
  \end{eqnarray}
where $\mathbb{K}_i$ is the Kernel of $\overline{\lambda_i}$ and $\mathbb{H}_i$ is the orthogonal of $\mathbb{K}_i$ relative to $\Omega_{\phi(x)}$ over $\phi(x)$.  Since the restriction $\overline{\lambda_i}'$ of $\overline{\lambda_i}$ to $\mathbb{H}_i$ is an isomorphism onto $\mathbb{E}_i$ and so  $\overline{\delta_i}\times\overline{\lambda_i}'$ is an isomorphism from $\phi(U)\times \mathbb{H}_i$ onto  $\phi(U)\times \mathbb{E}_i$. Thus, $(\Omega_i)=[\overline{\delta_i}\times\overline{\lambda_i}']^*\left((\Omega)_{| \phi(U)\times \mathbb{H}_i}\right)$ is a symplectic form on $\phi_i(U_i)\times \mathbb{E}_i)$ and so $\omega_i= \tau^*_i(\Omega)$ is a smooth symplectic form.\\
The end of the proof follows from Proposition \ref{P_FamilySymplecticSubmersion}.\\
\end{proof}

\begin{proof}[Proof of Corollary \ref{C_CoherentWeakSymplecticBanachManifold}]
According to the assumption of this Corollary, after applying Theorem \ref{T_CoherentWeakSymplecticBanachManifold}, the proof  will be completed if we prove that the $2$-form $\omega$ defined by the closed $2$-form $\omega_i$ is also closed and if $\omega$ is a closed $2$-form on the projective limit $M=\underleftarrow{\lim}M_i$ (each induced $2$ form $\omega_i$ induced on $M_i$ is closed).
Under the notations of the proof of Theorem \ref{T_CoherentWeakSymplecticBanachManifold}, we have
\begin{description}
\item[]
$\mathbb{E}_i=\mathbb{M}_i$  and $\mathbb{E}=\mathbb{M}$;
\item[]
$\lambda_i^j=T\delta_i^j$, $\overline{\ell_i^j }= \overline{\delta_i^j}$, $\lambda_i=T\delta_i $;
\item[]
$\tau_i=T\phi_i$, $\tau=T\phi$.
\end{description}
We can apply the context  of Lemma \ref{L_EProductOfBanachSpaces} and so if  $\mathbb{M}'_i=\ker \overline{\delta_i^j}$, then $\mathbb{M}_n$ is isomorphic to $\displaystyle\prod_{i=0}^n\mathbb{M}'_i$ and so $\mathbb{M}\equiv \displaystyle\prod_{i=0}^\infty\mathbb{M}'_i$. According to  Diagram \ref{eq_diagram3Disymplectic}  in our context, we have
\begin{eqnarray}
\label{eq_diag_Corollary3Disymplectic}
\xymatrix {
    \pi^{-1}(U) \ar[rr]^{\tau} \ar[dd] \ar[dr]^{\lambda_i}&& \phi(U)\times\displaystyle\prod_{l>i}\mathbb{M}'_l\times\displaystyle\prod_{l=0}^i \mathbb{M}'_l\ar[dr]^{\overline{\delta_i}\times \overline{\lambda_i}}\ar[dd] |!{[dl];[dr]}\hole \\
    & \pi_i^{-1}(U_i) \ar[rr]^{\tau_i} \ar[dd] && \phi_i(U_i)\times\displaystyle\prod_{l=0}^i \mathbb{M}'_l\ar[dd] \\
   U \ar[rr]^\phi |!{[ur];[dr]}\hole \ar[dr]^{{\delta}_i}&& \phi(U) \ar[rd]^{{\overline{\delta_i}}}\\
   & U_i \ar[rr]^{\phi_i} && \phi_i(U_i) \\
 }
 \end{eqnarray}
 and $\phi(U)$ is an open set in $\displaystyle\prod_{i=0}^\infty\mathbb{M}'_i$ and $\phi_i(U_i)$ is an open set of $\displaystyle\prod_{l=0}^i \mathbb{M}'_l$.  Thus, $\phi_i(U_i)$ is of type $\displaystyle\prod_{l=0}^i U'_l$ where $U'_l$ is an open set of $\mathbb{M}'_i$ and $\phi(U)$ is of type  $\displaystyle\prod_{l=0}^\infty U'_l$  where $U'_l$ is an open set of $\mathbb{M}'_l$ and with only a finite number of $l\geq i$  for which $U_l\not=\mathbb{M}'_l$.\\

 (1) Assume that $\omega$ is a projective limit of the sequence $(\omega_i)_{i \in \mathbb{N}}$.  As in the proof of Theorem \ref{T_CoherentWeakSymplecticBanachManifold}   let  $\Omega_i$ be  the form on $\phi_i(U_i)$ induced
  by $\omega_i$ and we denote by $\Omega$ the symplectic form on $\phi(U)$  induced by $\omega$ according to the context of Diagram \ref{eq_diag_Corollary3Disymplectic}.  If $\iota_l$ be the
  natural inclusion of $U'_l$ in $\displaystyle\prod_{l=0}^i\mathbb{M}'_l$,  we set $\overline{\Omega_l}=\iota_i^* \Omega_i$ for $l\geq i$. Note that $\overline{\Omega_l}$ does not depend on the choice of the integer $i\geq l$.  As $\Omega_i$ is closed , it follows that  $\overline{\Omega_l}$ is closed. Note that each subbundle $U'_l\times \mathbb{E}'_l$ is the tangent bundle of $U'_l$.  But, from the
 construction of $\omega$ ( and so $\Omega$ ), if $X_1$ and $X_2$ are vector fields on $\phi(U)$ which are tangent to $U'_{l_1}$ and $U'_{l_2}$ respectively, we have $\Omega(X_1, X_2)=0$ if $l_1\not=l_2$ and $\Omega(X_1, X_2)=\overline{\Omega_l}(X_1, X_2)$ if $=l_2=l_1=l$. This implies that $\Omega$ is closed.\\

(2) Assume that $\omega$ is a symplectic form such that $T_x\delta_i:T_xM\to T_{x_i}M_i$ is a symplectic   submersion. Then from Theorem \ref{P_PropertiesWeakIsometricLinear2Forms} (2), $\omega$ induces a  non degenerate $2$-form $\omega_i$ on $M_i$.

Again let $\Omega$ (resp. $\Omega_i$) be the $2$-form on $\phi(U)$ (resp. $\phi_i(U_i)$)  according to the context of Diagram \ref{eq_diag_Corollary3Disymplectic}. We must show that each $\Omega_i$ is closed.  Since $\Omega$ is the projective limit of the sequence $(\Omega_i)_{i \in \mathbb{N}}$, according to Theorem \ref{T_CoherentWeakSymplecticBanachManifold} (2).  Thus, as previously, if $X_1$ and $X_2$ are vector fields on $\phi(U)$ which are tangent to $U'_{l_1}$ and $U'_{l_2}$ respectively, we have $\Omega(X_1, X_2)=0$ if $l_1\not=l_2$ and $\Omega(X_1, X_2)=\overline{\Omega_l}(X_1, X_2)$ if $=l_2=l_1=l$. Thus $\Omega_i=\iota_i^*\Omega$ if $\iota_i$ is the natural inclusion of $\phi_i(U_i)$ in $\phi(U)$. It follows that each $\omega_i$ is closed.\\
\end{proof}

\section{Weak symplectic forms on a submersive projective  sequence of reflexive Banach bundles and Darboux-Bambusi assumption}
\label{___WeakSymplecticFormsOnProjectiveSequenceReflexiveBanachBundles-Darboux-BambusiAssumption}

\subsection{Moser's method and Darboux-Bambusi Theorem}
\label{__MosersMethodAndDarbouxTheorem}


We recall  the  following  generalization of Moser's Lemma (see \cite{Pe}).\\

Let $M$ be a manifold modeled on a reflexive Banach space $\mathbb{M}$. Consider a weak symplectic form $\omega$  on $M$. Then $\omega^{\flat}:TM\to T^{\ast}M$ is an injective bundle morphism. According to section \ref{linearsymplectic}, we denote by $\widehat{T_{x}M}$ the
Banach space which is the completion of $T_{x}M$ provided with the norm
$||\;||_{\omega_{x}}$ associated to some norm $||\;||$ on $T_xM$. The Banach space
 $\widehat{T_{x}M}$ does not depend on this choice.
Then $\omega_{x}$ can be extended to a continuous bilinear map $\hat{\omega}_{x}$ on $T_{x}M\times\widehat{T_{x}M}$ and $\omega_{x}^{\flat}$ becomes an isomorphism from $T_{x}M$ to
$(\widehat{T_{x}M})^{\ast}$. We set
\[
\widehat{TM}=\bigcup_{x\in M}\widehat{T_{x}M}\;\text{ and }\;(\widehat
{TM})^{\ast}=\bigcup_{x\in M}(\widehat{T_{x}M})^{\ast}.
\]

\begin{theorem}[Moser's Lemma]
\label{T_MoserLemmaOnAReflexifBanachSpace}
Let $\omega$ be a weak symplectic form on a Banach manifold $M$ modeled on a reflexive Banach space $\mathbb{M}$. Assume
that we have the following properties:
\begin{enumerate}
\item[(i)]
There exists a neighbourhood $U$ of $x_{0}\in M$ such that
$\widehat{TM}_{|U}$ is a trivial Banach bundle whose typical fibre is the Banach space $(\widehat{T_{x_{0}}M},||\;||_{\omega_{x_{0}}})$;
\item[(ii)]
$\omega$ can be extended to a smooth field of continuous bilinear forms on \\
$TM_{|U}\times\widehat{TM}_{|U}$.
\end{enumerate}

Consider a family $\{\omega^{t}\}_{0\leq t\leq1}$ of closed
$2$-forms which smoothly depends on $t$ with the following properties:
\begin{enumerate}
\item[--]
$\omega^{0}=\omega$ and $\forall t \in \left[ 0,1 \right], \omega_{x_{0}}^{t}={\omega}_{x_{0}}$;
\item[--]
$\;\omega^{t}$ can be extended to a smooth field of continuous bilinear forms on $TM_{|U}\times\widehat{TM}_{|U}$.
\end{enumerate}

Then there exists a neighbourhood $V$ of $x_{0}$ such that each $\omega^{t}$
is a symplectic form on $V$ and there exists a family $\{F_{t}\}_{0\leq t\leq1}$ of diffeomorphisms $F_{t}$ from a neighbourhood
$V_{0}\subset V$ of $x_{0}$ to a neighbourhood  $F_{t}(V_{0})\subset V$  of $x_0$ such
that $F_{0}=\operatorname{Id}$ and $F_{t}^{\ast}\omega^{t}=\omega$, for all $0\leq t\leq1$.
\end{theorem}

\begin{proof}[sketch for more details see \cite{Pe}]
 Without loss of generality, we may assume that $U$ is an open
neighbourhood of $0$ in $\mathbb{M}$ and $\widehat{TM}_{|U}=U\times
\widehat{\mathbb{M}}$. Therefore, $U\times\widehat{\mathbb{M}}$ is a trivial
Banach bundle modeled on the Banach space $(\widehat{\mathbb{M}%
},||\;||_{\omega_{0}})$.
Since $\omega$ can be extended to a non-degenerate skew symmetric bilinear
form (again denoted $\omega$) on $U\times(\mathbb{M}\times\widehat{\mathbb{M}%
})$ then $\omega^{\flat}$ is a Banach bundle isomorphism from $U\times
\mathbb{M}$ to $U\times\widehat{\mathbb{M}}^{\ast}$.\newline We set
$\dot{\omega}^{t}=\frac{d}{dt}\omega^{t}$. Since each $\omega^{t}$ is closed
for $0\leq t\leq1$, we have :
\[
d\dot{\omega}^{t}=\frac{d}{dt}(d\omega^{t})=0
\]
and so $\dot{\omega}^{t}$ is closed. After shrinking U if necessary, from
the Poincar\'{e} Lemma, there exists a $1$-form $\alpha^{t}$ on $U$ such that
$\dot{\omega}^{t}=d\alpha^{t}$ for all $0\leq t\leq1$. In fact $\alpha_{t}$
can be given by
\[
\alpha_{x}^{t}=\int_{0}^{1}s.(\dot{\omega}_{sx}^{t})^{\flat}(x)ds.
\]
Since at $x=0$, $(\omega_{x_{0}}^{t})^{\flat}$ is an isomorphism from
$\mathbb{M}$ to $\widehat{\mathbb{M}}^{\ast}$, there exists a
neighbourhood $V$ of $0$ such that $(\omega_{x}^{t})^{\flat}$ is an
isomorphism from $\mathbb{M}$ to $\widehat{\mathbb{M}}^{\ast}$ for all $x\in
V$ and $0\leq t\leq1$. In particular, $\omega^{t}$ is a symplectic form on
$V$. Moreover $x\mapsto(\dot{\omega}_{x}^{t})^{\flat}$ is smooth and takes
values in $\mathcal{L}(\mathbb{M},\widehat{\mathbb{M}}^{\ast})$.  We set $X_{x}^{t}:=-((\omega_{x}^{t})^{\flat
})^{-1}(\alpha_{x}^{t})$. It is a well defined time dependent vector field and
let $\operatorname{Fl}_{t}$ be the flow generated by $X^{t}$ defined on some neighbourhood
$V_{0}\subset V$ of $0$. As for all $t\in\lbrack0,1]$,  $\dot{\omega
}_{x_{0}}^{t}=0$, then $X_{x_{0}}^{t}=0$. Thus, for all $t\in\lbrack0,1]$,
$F_{t}(x_{0})=x_{0}$ . As classically, we have
\[
\frac{d}{dt}\operatorname{Fl}_{t}^{\ast}\omega^{t}=\operatorname{Fl}_{t}^{\ast}(L_{X^{t}}\omega^{t}%
)+\operatorname{Fl}_{t}^{\ast}\frac{d}{dt}\omega^{t}=\operatorname{Fl}_{t}^{\ast}(-d\alpha^{t}+\dot{\omega
}^{t})=0.
\]
Thus $\operatorname{Fl}_{t}^{\ast}\omega^{t}=\omega$.
\end{proof}

Now as a Corollary of Theorem \ref{T_MoserLemmaOnAReflexifBanachSpace}, we obtain the Bambusi's
version of Darboux Theorem (\cite{Bam}, Theorem 2.1).

\begin{theorem}
[ Darboux-Bambusi Theorem]\label{T_localDarboux} Let $\omega$ be a weak symplectic
form on a Banach manifold $M$ modelled on a reflexive Banach space
$\mathbb{M}$. Assume that the assumptions (i) and (ii) of Theorem
\ref{T_MoserLemmaOnAReflexifBanachSpace} are satisfied. Then there exists a chart $(V,F)$ around
$x_{0}$ such that $F^{\ast}\omega_{0}=\omega$ where $\omega_{0}$ is the
constant form on $F(V)$ defined by $(F^{-1})^{\ast}\omega_{x_{0}}$.\\
\end{theorem}

\begin{definition}
The chart $(V,F)$ in Theorem \ref{T_localDarboux} will be called a Darboux chart
around $x_{0}$.
\end{definition}

\subsection{Projective sequence of   weak symplectic bundle reflexive Banach bundle with Darboux-Bambusi assumptions}${}$\\
Let $\mathbb{E}=\underleftarrow{\lim}\mathbb{E}_{i}$ be a projective limit of a
projective sequence of reflexive Banach spaces $\left(   \mathbb{E}_{i},\overline{\lambda_i^j}\right)  _{j\geq i}
$.  We can  provide  each Banach space $\mathbb{E}_i$ with a norm $\left\Vert\;\right\Vert _{i}$  such
that $||\overline{\lambda_i^{i+1}}||^{\operatorname{op}}_i\leq 1$ for $i\in \mathbb{N}$.

We consider a sequence $\left(  \omega_{i}\right)_{i \in \mathbb{N}}$ of weak symplectic forms on $\mathbb{E}_i$ and
 let $\omega_{i}^{\flat}:\mathbb{E}_{i}\to \mathbb{E}_{i}^{\ast}$ be the associated bounded linear operator. According to notations in Remark \ref{indnorm}, we consider the norm $||u||_{\omega_{i}}=||\omega_{i}^{\flat}(u)||_{i}^{\ast}$ where $||\;||_{i}^{\ast}$ is the
canonical norm on $\mathbb{E}_{i}^{\ast}$ associated to $||\;||_{i}$.  We have seen that
the inclusion of the Banach space $(\mathbb{E},||\;||_{i})$ in the normed space  $(\mathbb{E}_{i},||\;||_{\omega_{i}})$ is continuous and we have denoted by  $\widehat{\mathbb{E}}_{i}$ the Banach space which is the completion of $(\mathbb{E}_{i},||\;||_{\omega_{i}})$. Recall that from Remark \ref{indnorm}, the Banach space  $\widehat{\mathbb{E}}_{i}$ does not depend on the choice of the norm $||\;||_i$ on $\mathbb{E}_i$.
According to section \ref{linearsymplectic} (before  Remark {indnorm}), $\omega_{i}^{\flat}$ can be extended to a symplectic  submersion between $\hat{\mathbb{E}}_{i}$ and $\mathbb{E}_{i}^{\ast}$. Moreover,
$\omega_{i}^{\flat}$ is an isomorphism from $\mathbb{E}_{i}$ to $\widehat
{\mathbb{E}}_{i}^{\ast}$.

\begin{lemma}
\label{L_wideEiProjective}
${}$
\begin{enumerate}
\item
The sequence $(\widehat{\mathbb{E}}_{i}^{*})_{i\in \mathbb{N}}$ is a projective sequence of Banach spaces and so
 $\widehat{\mathbb{E}}^{*}=\underleftarrow{\lim}\widehat{\mathbb{E}}_{i}^{*}$ is well defined. Moreover, if $\overline{\lambda_i^j}$ is surjective and its kernel is split,
then the bonding map  $\widehat{\overline{\lambda_i^j}}=\widehat{\mathbb{E}}^*_j\to \widehat{\mathbb{E}}^*_i$  also satisfies this assumption.
\item
The projective limit $\omega^\flat=\underleftarrow{\lim}\omega_i^\flat$ is well defined and is an isomorphism from $\mathbb{E}$ to $\widehat{\mathbb{E}}^{*}$
\end{enumerate}
\end{lemma}

\begin{proof}
(1) It is sufficient to show that $\overline{\lambda_i^{j}}$ and $\omega_i^\flat$ give rise to a  map $\widehat{\overline{\lambda_i^j}}$ from $ \widehat{\mathbb{E}}_{j}^{\ast}$ into $ \widehat{\mathbb{E}}_{i}^{\ast}$ and if $\overline{\lambda_i^{j}}$ is surjective and with a split kernel so is $\widehat{\overline{\lambda_i^j}}$. Indeed since $\omega_{i}^{\flat}$ is an isomorphism from $\mathbb{E}_{i}$ to $\widehat{\mathbb{E}}_{i}^{\ast}$, the bonding map $\widehat{\overline{\lambda_i^j}}=\omega_i^\flat\circ \overline{\lambda_i^j}\circ (\omega_i^\flat)^{-1}$ satisfied the announced properties in (1).\\
(2) is obvious.\\
\end{proof}
\newline

Now we consider a reduced  projective  sequence $\left(  E_{i},\pi_{i},M_{i}\right) _{\underleftarrow{i}}  $ of Banach vector bundles where the typical fibre $\mathbb{E}_{i}$ is reflexive. The projective limit $E=\underleftarrow{\lim}E_{i}$ has a structure Fr\'echet bundle
over $M=\underleftarrow{\lim}M_{i}$ with typical fibre $\mathbb{E}=\underleftarrow{\lim}\mathbb{E}_{i}$ (cf.  Proposition \ref{P_ProjectiveLimitOfBanachVectorBundles}).

Consider a  sequence $\left(  \omega_{i}\right) _{i _in \mathbb{N}}$ of compatible weak symplectic forms $\omega_i$ on $E_{i}$. According to the previous notations, since $\mathbb{E}_{i}$ is reflexive, we denote by
$(\widehat{E_{i}})_{x_{i}}$ the Banach space which is the completion of
$(E_{i})_{x_{i}}$ provided with the norm $||\;||_{(\omega_{_{i}})_{x_{i}}}$.
Then $(\omega_{i})_{x_{i}}$ can be extended to a continuous bilinear map
$(\hat{\omega}_{i})_{x_{i}}$ on $(E_{i})_{x_{i}}\times(\widehat{E_{i}}%
)_{x_{i}}$ and $(\omega_{i})_{x_{i}}^{\flat}$ becomes an isomorphism from
$(E_{i})_{x_{i}}$ to $(\widehat{E_{i}})_{x_{i}}^{\ast}$. We set
\[
\widehat{E_{i}}=\bigcup_{x_{i}\in M_{i}}(\widehat{E_{i}})_{x_i},\;\;\;\; \widehat{E_{i}^{\ast}}=\bigcup_{x_{i}\in M_{i}}(\widehat{E_{i}^{\ast}})_{x_{i}}
\]
According to the assumption of Theorem \ref{T_localDarboux}   we introduce the following terminology:

\begin{definition}
\label{D_DarbouxBambusiProjective}
Let $\left(  E_{i},\pi_{i},M_{i} \right) _{\underleftarrow{i}}$ be a reduced  projective sequence of
 Banach bundles whose typical fibre $\mathbb{E}_{i}$ is reflexive. Consider a  sequence $\left(  \omega_{i}\right)_{i _in \mathbb{N}} $ of compatible weak symplectic forms $\omega_{i}$ on $E_{i}$. We say that the sequence $\left(  \omega_{i}\right)  _{i\in\mathbb{N}}$ satisfies the Bambusi-Darboux assumption around $x^{0}\in M$ if there exists a projective limit chart $U=\underleftarrow{\lim}U_{i}$ around $x^{0}$ such that:
\begin{description}
\item[(i)]
for each $i\in \mathbb{N}$, $(\widehat{E}_{i})_{|U_{i}}$ is a trivial Banach
bundle;
\item[(ii) ]
for each $i\in \mathbb{N}$, $\;\omega_{i}$ can be extended to a smooth field
of continuous bilinear forms on $(E_{i})_{|U_{i}}\times(\widehat{E}_{i})_{|U_{i}}$.
\end{description}
\end{definition}

Under these assumptions we have:
\begin{proposition}
\label{P_DarbouxBambusiProjectiveProperties}
Consider a  sequence $\left(  \omega_{i}\right) _{i _in \mathbb{N}}$ of compatible  symplectic forms $\omega_{i}$ on $E_{i}$ which satisfies the Bambusi-Darboux assumption around $x^{0}\in M$. Then we have the following properties:
\begin{enumerate}
\item
The projective limit $\widehat{E}^*_{|U}=\underleftarrow{\lim
}\widehat{E_{i}^*}_{| U_{i}}$ is well defined and is a trivial Fr\'echet
bundle with typical fibre $\widehat{\mathbb{E}}=\underleftarrow{\lim
}\widehat{\mathbb{E}}_{i}$.
\item
The sequence $\left(  \omega_{i}^{\flat}\right)$ of isomorphisms from ${E_i}_{| U_i}$
to ${\widehat{E^*_i}}_{| {U}_i}$ induces an isomorphism from $E_{|U}$ to $\widehat{E}^\ast_{|U}$.
\end{enumerate}
\end{proposition}

\begin{proof}
(1) From our assumptions, for each $i$, we have a sequence of trivializations
$\widehat{\tau}_{i}:(\widehat{E}_{i})_{|U_{i}}\to U_{i}\times
\widehat{\mathbb{E}}_{i}$. Thus we obtain a sequence $\widehat{\tau}_{i}%
^{-1}:U_{i}\times\widehat{\mathbb{E}}_{i}^{\ast}\to(\widehat{E}%
_{i}^{\ast})_{|U_{i}}$ of isomorphisms of trivial bundles. Now, from the proof of  Lemma \ref{L_wideEiProjective},  we have the bonding map $\widehat{\overline{\lambda_{i}^{i}}}:$ $\widehat{\mathbb{E}}%
_{j}^{\ast}$ $\to$ $\widehat{\mathbb{E}}_{i}^{\ast}$ and by restriction to $U_j$ we have a bonding map
$\delta_{i}^{j}$ : $U_{j}\to U_{i}$. So we get a bundle
morphism $\delta_{i}^{j}\times\widehat{\overline{\lambda_{i}^{i}}}$ from $U_{j}\times\widehat{\mathbb{E}}_{j}^{\ast}$
to $U_{i}\times\widehat{\mathbb{E}}_{i}^{\ast}$.
Now  the map
\[
\widehat{\tau}_{i}^{-1}\circ(\delta_i^j\times\widehat{\overline{\lambda_{i}^{i}}})\circ \widehat{\tau}_j
\]
is a bonding map for the projective sequence of trivial bundles $\left(  (\widehat{E}_{i}^{\ast})_{|U_{i}}\right) _{\underleftarrow{N}}$. Therefore the projective limits $\widehat{\tau}=\underleftarrow{\lim}\widehat{\tau}_{i}$ and $\widehat{E}^{\ast}_{|U}=\underleftarrow{\lim}(\widehat{E}_{i}^{\ast})_{|U_{i}}$ are well
defined and $\widehat{\tau}$ is a Fr\'{e}chet isomorphism bundle from
$U\times\widehat{\mathbb{E}^*}$ to $\widehat{E}^*_{|U}$, which ends the proof of (1).\\

(2) At first, from Proposition \ref{P_PropertiesWeakIsometricLinear2Forms}, then $\omega=\underleftarrow{\lim}\omega_i$ is a $2$-form on $E$.
 From our assumption, since for each $i\in \mathbb{N}$ we
can extend $\omega_{i}$ to a bilinear onto $(E_{i})_{|U_{i}}\times(\widehat
{E}_{i})_{|U_{i}}$, this implies that $\omega_{i}^{\flat}$ is an isomorphism
from $(E_{i})_{|U_{i}}$ to $(\widehat{E}_{i}^{\ast})_{|U_{i}}$. Consider the
sequence of bonding maps $\left( \widehat{\overline{\lambda_{i}^{i}}} \right) _{i \in \mathbb{N}}$ for the projective sequence $\left(  (\widehat{E}_{}^{\ast})_{|U_{i}}\right) _{\underleftarrow{i}}$ previously defined. Then we have the following commutative diagram:
\[
 \xymatrix {
U_j\times\mathbb{E}_{j}\ar[r]^{\tau_j^{-1} } \ar[d]_{ \delta_i^j\times {\ell}_i^j}
&({E}_{j})_{|U_j}  \ar[r]^{\omega_j^\flat}\ar[d]^{\ell_i^j}&(\widehat{E}^*_{j})_{|U_j} \ar[r]^{\widehat{\tau}_{j}}\ar[d]^{\widehat{\ell}_i^j} & U_j\times\widehat{\mathbb{E}}^*_{j}\ar[d]^{\delta_i^j\times\widehat{{\ell}}_i^j}\\
U_i \times\mathbb{E}_{i}  \ar[r]_{{\tau}_{i}^{-1}} &
(E_i)_{|U_i}  \ar[r]_{\omega_{i}^{\flat}} &
(\widehat{E}^*_{i})_{|U_i}  \ar[r]_{\widehat{\tau}_{i}} & U_i\times\mathbb{\widehat{\mathbb{E}}}^*_{i}\\
}
\]
It follows that the projective limit $\omega^{\flat}=\underleftarrow{\lim}\omega_{i}^{\flat}$ is well defined and is an isomorphism from $E_{|U}$
to $\widehat{E}_{|U}^{\ast}$.
\end{proof}

\section{Problem of existence of Darboux charts on a strong reduced projective sequence of Banach manifolds}
\label{___ProblemOfExistenceOfDarbouxChartsonAStrongReducedProjectiveSequenceOfBanachManifolds}
\subsection{Conditions of existence of Darboux charts}

Let $ \left( M_{i},\delta_i^j \right) _{j\geq i}$ be a  submersive or decreasing projective sequence of Banach manifolds where $M_{i}$ is modeled on a Banach space $\mathbb{M}_{i}$. We first apply the previous results for $E_i=TM_i$.

\begin{theorem}
\label{T_omega_nCoherentM}
${}$
\begin{enumerate}
\item
Consider a   sequence $\left(\omega_{i}\right) _{i \in \mathbb{N}} $ of compatible  weak symplectic forms $\omega_{i}$
on $M_{i}$. Then, for each $x\in M$, the projective limit $\omega_{x}^{\flat}=\underleftarrow{\lim}(\omega_{i})_{x_{i}}^{\flat}$ is
well defined and is an isomorphism from $T_{x}M$ to $(\widehat{T_{x}M})^{\ast
}$. Moreover $\omega_{x}(u,v)=\omega_{x}^{\flat}(u)(v)$ defines a smooth
weak symplectic form on $M$.
\item
Let $\omega$ be a symplectic form on  a submersive projective limit manifold  $M=\underleftarrow{\lim}M_i$. For all $i\in \mathbb{N}$,
assume that  the canonical projection $\delta_i: M\to M_i$ is  a symplectic   submersion. Then there exists a symplectic form  $\omega_i$ on $M_{i}$  such that $\delta_i^*\omega_i=\omega$ in restriction to $(\ker\delta_i)^\perp$ and the sequence $\left(  \omega_{i}\right) _{i \in \mathbb{N}} $ is a sequence of compatible   weak symplectic forms such that the weak symplectic form  which is the projective limit of $(\omega_i) _{i \in \mathbb{N}}$ on $M$ is exactly $\omega$.
\end{enumerate}
\end{theorem}

\begin{proof}
(1) Since $\omega(u,v)=\omega^\flat(u)(v)$, by application  of Proposition \ref{P_DarbouxBambusiProjectiveProperties}  to $E_{i}=TM_{i}$, we obtain that $\omega$ is non degenerate. The proof of that $\omega$ is closed is formally the same as in the proof of Corollary \ref{C_CoherentWeakSymplecticBanachManifold} (1).\\
(2) is a direct consequence of Corollary \ref{C_CoherentWeakSymplecticBanachManifold} (2).\\
\end{proof}
\newline
As in the Banach context, we introduce the notion of Darboux chart:
\begin{definition}
\label{D_DarbouxChartProjectiveLimit}
Let $\omega$ be a weak symplectic form on the direct limit $M=\underleftarrow{\lim}M_{i}$. We say that a chart $(V,\psi)$ around $x_{0}$ is a Darboux chart if $\psi^{\ast}\omega^{0}=\omega$ where $\omega^{0}$ is the constant form on $\psi(U)$ defined by $(\psi^{-1})^{\ast}\omega_{x_{0}}$.
\end{definition}

We have the following necessary and sufficient conditions of existence of Darboux charts on  a submersive  projective sequence of Banach manifolds:

\begin{theorem}
\label{T_equivDarbouxProjective}
Let $\left( M_{i},\delta_i^j\right)  _{j\geq i}$ be a  submersive  or decreasing  projective sequence of Banach manifolds where $M_{i}$ is modeled on a reflexive Banach space $\mathbb{M}_{i}$.
\begin{enumerate}
\item
Consider a  sequence $(\omega_{i}) _{i \in \mathbb{N}}$ of compatible  symplectic forms $\omega_{i}$ on $M_{i}$ and let
$\omega$ be the symplectic form which is the projective limit of $(\omega_{i}) _{i \in \mathbb{N}}$ on
$M=\underrightarrow{\lim}M_{i}$. Assume that the following property is satisfied:

\begin{description}
\item[\textbf{(D)}]
There exists a projective limit chart $(U=\underleftarrow{\lim}U_{i},\phi=\underleftarrow{\lim}\phi_{i})$ around $x^{0}$ such that, for each  $x^{0}_i=\delta_i(x^0)\in M_{i}$, then $(U_{i},\phi_{i})$ is a Darboux chart around $x^{0}_i$ for $\omega_{i}$.
\end{description}
Then $(U,\phi)$ is a Darboux chart around $x^{0}$ for $\omega$.

\item
Let $\omega$ be a weak symplectic form on a submersive projective limit
$M=\underleftarrow{\lim}M_{i}$ such that $\delta_i:M\to M_i$ is a symplectic   submersion. Assume that there exists a Darboux chart $(V,\phi)$ around $x^{0}$ in $M$.\\
If $\omega_{i}$ is the symplectic form on $M_{i}$ induced by $\omega$, then there exists a  projective limit chart
$(U=\underleftarrow{\lim}U_{i},\phi=\underrightarrow{\lim}\phi_{i})$ around
$x^{0}$ such that the property (D) is satisfied.\newline
\end{enumerate}
\end{theorem}

\begin{proof}
(1) Assume that the assumption \textbf{(D)} is true and that $ \left( M_{i} \right) _{i \in \mathbb{N}}$ is a reduced projective sequence of Banach manifolds. We fix some $x^0\in M$. We consider a projective limit chart $(U=\underleftarrow{\lim}U_{i},\phi
=\underleftarrow{\lim}\phi_{i})$ around $x^{0}$ such that, if $x^{0}_i=\delta_i(x^0)\in U_{i}$, then $(U_{i},\phi_{i})$ is a Darboux chart around $x^{0}_i$ for $\omega_{i}$. Now we have the following commutative diagram:
\begin{eqnarray}
\label{eq_diagram3DDarboux}
\xymatrix {
    \pi^{-1}(U) \ar[rr]^{T\phi} \ar[dd] \ar[dr]^{T\delta_i}&& \phi(U)\times\mathbb{M}\ar[dr]^{\bar{\delta}_i\times{\bar{\delta}}_i} \ar[dd] |!{[dl];[dr]}\hole \\
    & \pi_i^{-1}(U_i) \ar[rr]^{T\phi_i} \ar[dd] && \phi_i(U_i)\times \mathbb{M}_i \ar[dd] \\
    U \ar[rr]^\phi |!{[ur];[dr]}\hole \ar[dr]^{{\delta}_i}&& \phi(U) \ar[rd]^{{\bar{\delta}}_i} \\
    & U_i \ar[rr]^{\phi_i} && \phi_i(U_i) \\
  }
\end{eqnarray}
According to this diagram and modulo the diffeomorphisms $\phi$ and $\phi_i$,  we may assume that
\begin{itemize}
\item[--]
$U$ is an open neighbourhood of $x^0\equiv 0 \in\mathbb{M}$, and   $U_i$ is a neighbourhood of  $0\in \mathbb{M}_i$;
\item[--]
$\omega$ is a smooth  $2$-form on $U$ and $\omega_i$ is a constant $2$-form on $U_i$.
\end{itemize}

Now if $x=\underleftarrow{\lim}x_i\in U$, $u=\underleftarrow{\lim}u_i$ and $v=\underleftarrow{\lim}v_i$, since $\omega_i$  is constant on $U_i$ it follows that $(\omega_i)_{x_i}(u_i,v_i)$ is independent of $x_i\in U_i$; so the value
\[
\omega_x (u,v)=\underleftarrow{\lim}(\omega_i)_{x_i}(u_i,v_i)
\]
is independent of the point $x$, which ends the proof of (1).\\

(2) Let $\omega$ be a weak symplectic form on $M=\underleftarrow{\lim}M_{i}$ such that, for all $i\in \mathbb{N}$, $\delta_i:M\to M_i$ is a symplectic  submersion.
Assume that we have a Darboux chart $(U=\underleftarrow{\lim}U_{i},\phi
=\underrightarrow{\lim}\phi_{i})$ around $x^{0}$ for $\omega$.
Fix some $i\in \mathbb{N}$.   In the  context of  Diagram(\ref{eq_diagram3DDarboux}), we have $\mathbb{M}\equiv \mathbb{K}_i\times \mathbb{H}_i$ where $\mathbb{K}_i$ is the kernel of $T_0\delta_i$ and $\mathbb{H}_i$ is the orthogonal of $\mathbb{K}_i$ in $\mathbb{M}_i\equiv T_0M$ (cf. Diagram(\ref{eq_diagram3Di}) with, for all $i\in \mathbb{N}$, $E_i=TM_i$). Thus again, modulo the diffeomorphisms $\phi$ and $\phi_i$, we may assume that
\begin{itemize}
\item[--]
$x^0\equiv 0\in U\subset \mathbb{K}_i\times \mathbb{H}_i$ , $\;x_i^0\equiv 0\in U_i\subset  \mathbb{M}_i$;
\item[--]
$\omega$ is a constant $2$-form on $U$ and  $\omega_i$ is a smooth $2$-form on $U_i$.
\end{itemize}

Recall that the restriction of  $\overline{\delta_i}$  to $\mathbb{H}_i$  is an  isomorphism onto $\mathbb{M}_i$, thus we may also assume that $\mathbb{H}_i=\mathbb{M}_i$.
In this way,  we have   $\delta_i^*\omega_i=\omega$ in restriction to $\mathbb{H}_i= \mathbb{M}_i$. Thus, with our identification, $\omega_i$ is nothing but the restriction of $\omega$ to $U_i\times \mathbb{M}_i$ and so $\omega_i$ is a constant $2$-form on $U_i$ whose value is fixed by the restriction of $\omega$ to $\mathbb{M}_i$.
\end{proof}

\subsection{Problem of existence of Darboux chart in general}
\label{__ProblemOfExistenceOfDarbouxChartInGeneral}

In this subsection, we will explain why, even in the context of a  submersive projective sequence of  weak symplectic Banach manifolds which satisfies the assumption of Theorem \ref{T_localDarboux} , in general, there does not exist any Darboux chart for the induced symplectic form on the projective limit.\\

Let $  \left( M_{i},\delta_{i}^{j}) \right) _{j\geq i}$ be a  projective sequence of Banach manifolds where $M_{i}$ is modeled on a \emph{reflexive} Banach space $\mathbb{M}_{i}$.
Consider a  sequence $\left(  \omega_{i}\right)  _{i\in\mathbb{N}}$ of compatible  weak symplectic forms on $M_{i}$. Since $\mathbb{M}_{i}$ is reflexive, we denote by $\widehat{T_{x_{i}}M_{i}}$ the Banach space which is the completion of $T_{x_{n}}M_{i}$ provided with the norm $||\;||_{(\omega_{i})_{x_{i}}}$.
Then $(\omega_{i})_{x_{i}}$ can be extended to a continuous bilinear map
$(\hat{\omega}_{i})_{x_{i}}$ on $T_{x_{i}}M_{i}\times\widehat{T_{x_{i}}M_{i}}$
and $(\omega_{i})_{x_{i}}^{\flat}$ becomes an isomorphism from $T_{x_{i}}%
M_{i}$ to $(\widehat{T_{x_{i}}M_{i}})^{\ast}$. We set
\[
\widehat{TM_{i}}=\bigcup_{x_i\in M_i}\widehat{T_{x_{i}}M_{i}},\;\;\;\;\;\widehat{TM_{i}}^*=\bigcup_{x_i\in M_i}\widehat{T_{x_{i}}M_{i}}^*.
\]
Then by application of Proposition \ref{P_DarbouxBambusiProjectiveProperties}, we have:
\begin{proposition}
\label{P_DarbouxProjective}
Let $\left(  M_{i},\delta_{i}^{j}\right) _{j\geq i}$ be a reduced projective sequence of Banach manifolds whose
model is a reflexive Banach space $\mathbb{M}_{i}$. Consider a  sequence $\left(  \omega_{i}\right)$ of compatible   weak symplectic forms $\omega_{i}$ on $M_{i}$. Assume that we have the following assumptions
\footnote{These assumptions correspond to the Bambusi-Darboux assumptions in Definition \ref{D_DarbouxBambusiProjective}} at $x^0\in M=\underleftarrow{\lim}M_i$:
\begin{enumerate}
\item[(i)]
there exists a limit chart $(U=\underleftarrow{\lim}U_{i},\phi=\underleftarrow{\lim}\phi_{i})$ around $x^{0}$ such that $(\widehat{TM}_{i})_{|U_{i}}$ is a trivial Banach bundle.

\item[(ii)]
$\omega_{i}$ can be extended to a smooth field of continuous bilinear forms on $(TM_{i})_{|U_{i}}\times(\widehat{TM}_{i})_{|U_{i}}$ for all $i\in \mathbb{N}$.
\end{enumerate}

Then $\widehat{T^{\ast}M}_{|U}$ is a trivial bundle. If $\omega$ is the
symplectic form defined by the sequence $\left(  \omega_{i}\right)
_{i\in\mathbb{N}}$, then the morphism
\[
\omega^{\flat}:TM\to T^{\ast}M
\]
induces an isomorphism from $TM_{|U}$ to $\widehat{T^{\ast}M}_{|U}$.
\end{proposition}

Note that the context of Proposition \ref{P_DarbouxProjective} covers the
particular framework of projective limit of strong symplectic Banach manifolds
$\left( M_{i},\omega_{i} \right) _{i \in \mathbb{N}}$.\\

\textit{We will expose which arguments are needed to prove  a Darboux theorem in the context of reduced  projective sequence of Banach manifolds under the assumptions of Proposition \ref{P_DarbouxProjective}. In fact, we point out the problems that arise in establishing the existence
of a Darboux chart by Moser's method}.\\ 

\underline{\bf Case 1}. Assume that $M=\underleftarrow{\lim}M_i$ is a  reduced projective limit.\\
 Fix some point $a=\underleftarrow{\lim}a_{n}\in M$. In the context on
Proposition \ref{P_DarbouxProjective}, on the projective  limit chart $(U,\phi)$ around $a$, we can
replace $U$ by $\phi(U)$, $\omega$ by $\phi^{\ast}\omega$ on the open subset  $\phi(U)$ of the Fr\'echet
 space $\mathbb{M}$. Thus, if $\omega^{0}$ is the constant form
on $U$ defined by $\omega_{a}$, we consider the $1$-parameter family
\[
\omega^{t}=\omega^{0}+t\overline{\omega},\text{ with }\overline{\omega}%
=\omega-\omega^{0}.
\]
Since $\omega^{t}$ is closed and $\mathbb{M}$ is a Fr\'echet space, by
\cite{KrMi} Lemma 33.20, there exists a neighbourhood $V\subset U$ of $a$ and
a $1$-form $\alpha$ on $V$ such that $d\alpha=\overline{\omega}$ which is
given by
\[
\alpha_{x}:=\int_{0}^{1}s.\overline{\omega}_{sx}(x,\;)ds.
\]
Now, for all $0\leq t\leq1$, $\omega_{x_{0}}^{t}$ is an isomorphism from
$T_{a}M\equiv\mathbb{M}$ onto $\widehat{T_{a}M}\equiv\widehat{\mathbb{M}%
}^{\ast}$. In the Banach context, using the fact that the set of invertible
operators is open in the set of operators, after restricting $V$, we may
assume that $(\omega^{t})^{\flat}$ is a field of isomorphisms from
$\mathbb{M}$ to $\widehat{\mathbb{M}}^{\ast}$.
 \textit{Unfortunately, this result is not true in the Fr\'echet setting}.
Therefore, the classical proof does not work in
this way in general. \\

\underline{\bf Case 2}. Assume that  $M$ is a submersive projective limit.\\
According to Theorem \ref{T_omega_nCoherentM}, assume that the canonical projection $\delta_i:M\to M_i$ is a symplectic   submersion, for all $i\in \mathbb{M}$. Then  $\omega$ induces a symplectic form $\omega_i$ on
$M_{i}$.  Therefore, for each $i$, let $\alpha_{i}$ be the $1$-form
 induced by $\alpha$ on $\phi_{i}(U_{i}\cap V)$. Then we have $\omega_{i}=d\alpha_{i}$ and also
\[
(\alpha_{i})_{x_{i}}=\int_{0}^{1}s.(\overline{\omega}_{i})_{sx_{i}}%
(x_{i},\;)ds
\]
where $\bar{\omega}_{i}=\omega_{i}-\omega_{i}^{0}$ is associated to the
$1$-parameter family $\omega_{i}^{t}=\omega^{i}+t\bar{\omega}_{i}$. We are
exactly in the context of the proof of Theorem \ref{T_MoserLemmaOnAReflexifBanachSpace} and so the
local flow $\operatorname{Fl}^{X_i}_{t}$ of $X_{i}^{t}=((\omega_{i}^{t})^{\flat})^{-1}(\alpha_{i})$ is a local diffeomorphism from a neighbourhood $W_{i}$ of $a_i$ in
$V_{i}$ and, in this way, we build a Darboux chart around $a_{i}$ in $M_{i}$.
Therefore, after restricting each $W_{i}$, if necessary, assume that:
\begin{description}
\item[\textbf{(PLDC)}](projective limit Darboux chart)
We have a projective sequence of such open sets $\left( W_{i} \right) _{n\in \mathbb{N}}$, then on $W=\underleftarrow
{\lim}W_{i}$, the family of local diffeomorphisms $F^{t}=\underleftarrow
{\lim}F_{i}^{t}$ is defined on $W$.
\end{description}
Recall that $\omega^{\flat}=\underleftarrow{\lim}\omega_{i}%
^{\flat}$ and $\omega^{\flat}$ is an isomorphism. Thus according to the
previous notations, we have a time dependent vector field
\[
X^{t}=((\omega^{t})^{\flat})^{-1}(\alpha)
\]
and again, we have $L_{X^{t}}\omega^{t}=0$. Of course, if the \textbf{(PLDC)}
assumption on $\left( W_{n} \right) _{n\in \mathbb{N}}$ is true, then $X^{t}=\underleftarrow{\lim}X_{i}^{t}$.
So we obtain a Darboux chart as in the Banach context. Note that, in this
case, we are in the context of Theorem \ref{T_equivDarbouxProjective} .

\begin{remark}
\label{R_pbDarbouxProjectiveChart}
In fact, under  the assumption \emph{\textbf{(PLDC)}}, the flow $\operatorname{Fl}_{t}$ is the local flow (at
time $t\in [0,1]$) of $X^{t}=\underrightarrow{\lim}X_{i}^{t}$ where $X_{i}^{t}=((\omega_{n}^{t})^{\flat})^{-1}
(\alpha_{i})$ (with the previous notations).
Unfortunately, according to Remark \ref{R_ProjectiveLimitTopologyFrechettoplogy}, outside particularity special cases, the "Darboux chart" assumption is not true in general,  since, in general,
\[
\displaystyle\bigcap_{j\geq i_0}\delta_i^j(W_j)
\]
is not an open neighbourhood of $x^0_{i_0}$.

\end{remark}

\underline{Consider  again the context of case 2.}

Fix some norm $||\;||_i$  on $T_{x^0_i}M_i$ for all $i\in \mathbb{N}$.  Assume there exists $K>0$ such  that
$||((\omega_{n}^{t})_{x^0_i}^{\flat})||_i^{\operatorname{op}}\leq K$, for all $t\in [0,1]$ and for all $i\in \mathbb{N}$. Then, according to Theorem \ref{T_UniformlyBoundedProperties}, there exists an open  neighbourhood $W$  of $a\in M$ such that $(\omega_{n}^{t})^{\flat}$ is uniformly bounded on $W$ and so  the same is true for $((\omega_{n}^{t})^{\flat})^{-1}$. It follows that    the time dependent vector field  $X^t$  defined in Remark \ref{R_pbDarbouxProjectiveChart} satisfies the assumption of Theorem \ref{T_SolutionODEOnFrechetSpaces}
and so the assumption \textbf{(PLDC)} will be satisfied.\\
Conversely  if $\omega$ is a symplectic form on $M$ such that $\delta_i :M\to M_i$ is a symplectic  submersion, then we can apply the previous arguments. Thus we have:
 \begin{theorem}
\label{T_Uniformly BoundedSymplectic}
 	Let $M=\underleftarrow{\lim}M_i$  be a submersive projective limit of reflexive Banach manifolds $M_i$ and $(\omega_i)_{i \in \mathbb{N}}$ a compatible  sequence of symplectic forms (resp. $\omega$ a symplectic form on $M$ such that $\delta_i:M\to M_i$ is a symplectic  submersion).\\
If the assumptions of Proposition \ref{P_DarbouxProjective} are satisfied around some point $a=\underleftarrow{\lim}a_i \in M$ and if
  \[
  \exists K>0:\; \forall t \in [0,1],\, \forall i \in \mathbb{N},\;
  ||((\omega_{n}^{t})_{a_i}^{\flat})||_i^{\operatorname{op}}\leq K
  \]
then there exists a Darboux chart around $a$.
 \end{theorem}
 \begin{remark}
 \label{R_thKumarThEftekharinasab}
  A theorem of existence of a Darboux chart in the context of projective limit of Banach manifolds was firslty proved by Kumar (\cite{Ku2}, Theorem 5.1). The sufficient condition required in this Theorem  for  the existence of such Darboux Chart  also implies the validity of "Darboux chart assumption". More precisely, under the previous notations, it is  assumed  that
\[
 \exists K>0:\; \forall t \in [0,1],\, \forall i \in \mathbb{N},\;
  ||((\omega_{n}^{t})^{\flat})^{-1}(\alpha_i)||_i\leq K.
\]

 Note that the context of Kumar's Theorem is the same as in Theorem \ref{T_Uniformly BoundedSymplectic}, except that the  previous last condition is stronger than the last condition  of Theorem \ref{T_Uniformly BoundedSymplectic}.\\
 The big problem of such results is that, without very particular case (cf \cite{Ku3}), to our knowledge, there exists no general situation in which such a result can be applied.
  \end{remark}

\subsection{Examples and contre-example about the existence of a projective limit of Darboux charts}
\label{_ExamplesContreExamplePeojectiveLimitDabouxChart}

\begin{example}
\label{Ex_ProjectiveDarbouxChartOnLpSM}\normalfont
According to \cite{Pe} section  4,  the set $\mathsf{L}_k^p(\mathbb{S}^1,M)$  of Sobolev loops of class $L_p^k$ has a Banach structrue manifold and if  where $(M,\omega)$ is a symplectic manifold, we can provide  $\mathsf{L}_k^p(\mathbb{S}^1,M)$ with a weak symplectic form $\Omega_k$ and around any $\gamma\in \mathsf{L}_k^p(\mathbb{S}^1,M)$, we have a Darboux chart (cf. \cite{Pe} Theorem 32). Moreover,  $\mathsf{L}_k^2(\mathbb{S}^1,M)$  is a Hilbert space and $\Omega_k$ is a strong symplectic form. If we denote by  $\mathsf{L}^\infty(\mathbb{S}^1,M)$ the set of smooth loops in $M$, we have  $\mathsf{L}^\infty(\mathbb{S}^1,M)=\underleftarrow{\lim}\mathsf{L}_k^2(\mathbb{S}^1,M)$  and this space is a $\mathsf{ILH}$-manifold. It is easy to see that the sequence of forms $(\Omega_k)_{k\in \mathbb{N}}$ are compatible  and since the projective sequence $\left(\mathsf{L}_k^2(\mathbb{S}^1,M)\right)_{k\in \mathbb{N}}$ is reduced,  we get a weak symplectic form $\Omega=\underleftarrow{\lim}\Omega_k$ on  $\mathsf{L}^\infty(\mathbb{S}^1,M)$. In fact, $\Omega$ can be defined directly in the same way as $\Omega_k$ on each  $\mathsf{L}_k^p(\mathbb{S}^1,M)$. \\
When $M=\mathbb{R}^{2m}$, consider  the canonical (linear) Darboux form $\omega$ on $\mathbb{R}^{2m}$. Then we have a global Darboux chart for $\Omega$ on $\mathsf{L}^\infty(\mathbb{S}^1,\mathbb{R}^{2m})$ (cf. \cite{Ku3}). Of course, since we also have a global Darboux chart on each $\mathsf{L}_k^2(\mathbb{S}^1,\mathbb{R}^{2m})$, we then get an example of projective limit of Darboux charts.
\end{example}

\begin{example}
\label{Ex_existDarbouxProjectiveLimit}\normalfont
Let  $\left( \mathbb{M}_i\right)  _{i\in \mathbb{N}}$ be a sequence of Banach spaces.  Consider the submersive projective sequence of Banach spaces $\left( \overline{\mathbb{M}}_i=\displaystyle\prod_{k=1}^i \mathbb{M}_k\right)  _{i\in \mathbb{N}^\ast}$ of Banach spaces where $\bar{\delta}_i^j:  \overline{\mathbb{M}}_j\to \overline{\mathbb{M}}_i$ is the canonical projection. Then the projective limit  $\overline{\mathbb{M}}$ is the product $\displaystyle\prod_{k=1}^\infty \mathbb{M}_k$. On $\overline{\mathbb{M}}$ the projective limit topology   is the product topology and it is also the topology of Fr\'echet manifold.\\
Now,  assume  that on each $\mathbb{M}_k$ we have  a weak symplectic form $\omega_k$ such that, for some $\bar{x}= \underleftarrow{\lim}(x_1,\dots,x_n)\in \overline{\mathbb{M}}$, each symplectic form $\omega_k$ satisfies  the assumptions (i) and (ii) of Theorem \ref{T_MoserLemmaOnAReflexifBanachSpace} at $x_k$ and  for all $k\in \mathbb{N}^\ast$. Then from  this Theorem, around the point $x_k\in \mathbb{M}_k$, we have a Darboux chart $(V_k, F_k)$. \\
 For any $\bar{x}_n:=(x_1,\dots,x_n)\in \overline{\mathbb{M}}_n$ and $\bar{u}_n:=(u_1,\dots,u_n)$, $\bar{v}_n=(v_1,\dots,v_n)$ in $T_{\bar{x}_n}\overline{\mathbb{M}}_n$ we define the $2$ form
\[
\bar{\omega}_n(\bar{u}_n,\bar{v}_n):=\displaystyle\sum_{k=1}^n\omega_k(u_k,v_k).
\]
Then $\bar{\omega}_n$ is also a  weak symplectic form on $\overline{\mathbb{M}}_n$ and it is easy to see that $(\overline{V}_n, \overline{F}_n)$ is a Darboux chart for $\bar{\omega}_n$ around $\bar{x}_n$. Now it is clear that the sequence $\left( \bar{\omega}_n\right)  _{n\in \mathbb{N}}$ of weak symplectic forms are compatible  and so give rise to a weak symplectic form $\bar{\omega}$ on $\overline{\mathbb{M}}$. Then $(\overline{V}=\underleftarrow{\lim}\overline{V}_n, \underleftarrow{\lim}\overline{F}_n)$ is a Darboux chart around $\bar{x}:=\underleftarrow{\lim}\bar{x}_n$ if  $V$ is an open set if and only if $\overline{V}_n=\overline{M}_n$ for any $n\in \mathbb{N}$ outside a finite subset $J\subset \mathbb{N}$. Such a situation occurs for instance in the following contexts: 
\begin{enumerate}
\item$ \omega_k$ is a linear Darboux form on the Banach space $\mathbb{M}_k$ for all $k\in \mathbb{N}$ eventually outside of finite set $J$ (cf. section \ref{linearsymplectic}).

\item$ \omega_k$ is a weak linear  symplectic form on the reflexive Banach space $\mathbb{M}_k$ for all $k\in \mathbb{N}^\ast$  eventually outside of finite set $J$ (cf. \cite{CaPe} Proposition B.3 Point (3))
 \item $\mathbb{H}$ is a separable infinite-dimensional real Hilbert space and we consider:
\begin{itemize}
\item[--]
$\mathbb{M}_k=\mathbb{H}$ for each integer  $k\in \mathbb{N}$;
\item[--]
$S_k:\mathbb{H}\to \mathbb{H}$ is  a compact operator with dense
range, but proper subset of $\mathbb{H}$, which is self adjoint and positive\footnote{ such operators $S_k$ exist since  the Hilbert space $\mathbb{H}$ is separable and infinite-dimensional} (such an operator is injective)
\item[--]
 $\hat{\omega}$ a linear Darboux form on $\mathbb{H}$ and $\omega_k=S_k^*\hat{\omega}$ for at most a finite number of integers and otherwise $S_k=Id_{\mathbb{H}}$	
\end{itemize}
\end{enumerate}
\end{example}

From the example of \cite{Ma2}, we can obtain the following example for which \textbf{there is no Darboux chart} on a submersive projective limit of symplectic Banach manifolds:

\begin{example}
\label{Ex_NoDarbouxChart} \normalfont Let $\mathbb{H}$ be a separable infinite-dimensional real Hilbert space  endowed with its
inner product $<\;,\;>$. If $g$ is a weak Riemannian metric on $\mathbb{H}$, we may use the trivialization $T\mathbb{H}=\mathbb{H}\times\mathbb{H}$ to define a weak symplectic form $\omega$ in the following way (\cite{Ma2}):
\[
2\;\omega_{(x,e)}((u,v),(u',v'))=D_{x}g_{x}(e,u).u'%
-D_{x}g_{x}(e,u').u+g_{x}(v',u)-g_{x}(v,u').
\]
Then the operator $\omega_{(x,e)}^\flat: T_{(x,e)}\mathbb{H}\times\mathbb{H}\to T_{(x,e)}^*\mathbb{H}\times\mathbb{H}$ can be written as a matrix of type
\[
\displaystyle\frac{1}{2}\begin{pmatrix}
\Gamma_{(x,e)}& -g_x^\flat\\
g_x^\flat&0\\
\end{pmatrix}
\]
Since $g_x^\flat$ is always injective by assumption, it follows that  $\omega_{(x,e)}^\flat$ is always injective and is surjective if and only $g_x^\flat$ is so. It follows that  if $\Sigma$ is the set of points $x\in \mathbb{H}$ where $g_x^\flat$ is  not surjective, then the set of points $(x,e)\in T\mathbb{H}$ where $\omega_{(x,e)}$ is not a strong symplectic form is precisely $\Sigma\times \mathbb{H}$.\\

As at the end of the above Example, let  $S:\mathbb{H}\to\mathbb{H}$ be a compact operator with dense
range, but proper subset of $\mathbb{H}$, which is self adjoint and positive.
Given a fixed $a\in\mathbb{H}$, then $A_{x}=||x-a||^{2}Id_{\mathbb{H}}+S$ is a
smooth field of bounded operators of $\mathbb{H}$ which is an isomorphism for
all $x\not =a$ and $A_{a}(\mathbb{H})\not =\mathbb{H}$ but  $A_{a}(\mathbb{H})$
 is dense in $\mathbb{H}$ (cf. \cite{Mars2}). Then $g_{x}(e,f)=<A_{x}(e),f>$ is a weak Riemaniann metric and
the associated symplectic form $\omega_{(x,e)}$ is not a strong symplectic form if and only if $(x,e)$ belongs to $\{a\}\times \mathbb{H}$
and, in this case, the range of $\omega_{(x,e)}^\flat$ is dense in $T_{(x,e)}^*(\mathbb{H}\times\mathbb{H})\equiv \mathbb{H}\times\mathbb{H}$.\\

For each $k\in \mathbb{N}^\ast$ and any $x\in \mathbb{H}$ we set
\[
(A_{k})_{x}=||x-\frac{a}{k}||^{2}Id_{\mathbb{H}}+S.
\]
We consider the Hilbert space $\displaystyle\overline{\mathbb{H}}_n =\prod_{k=1}^n \mathbb{H}_k$  where  $\mathbb{H}_k=\mathbb{H}$ and provided with the inner product
\[
<(x_1,\dots,x_n),(y_1,\dots,y_n)>_n=\displaystyle\sum_{k=1}^n <x_k,y_k>
\]
 As in the previous example, we identify $\overline{\mathbb{H}}_n$ with $\overline{\mathbb{H}}_n\times\{0\}$ in  $\overline{\mathbb{H}}_{n+1}$.  From now on, we will use the notations introduced in  Example \ref{Ex_existDarbouxProjectiveLimit}.\\
For any $\bar{x}_n=(x_1,\dots,x_{n})\in\overline{\mathbb{H}}_{n}$, we set
\[
(\ell_{n})_{\bar{x}_n}=\left((A_1)_{x_1},\dots, (A_n)_{x_n}\right) .
\]
We denote by $g_{n}$ the Riemannian metric on $\overline{\mathbb{H}}_{n}$ defined by
\[
(g_{n})_{\bar{x}_n}(\bar{u}_n,\bar{v}_n)=<(\ell_{n})_{(x_1,\dots,x_n)}(\bar{u}_n),\bar{v}_n>_n
\]
for all  $ \bar{u}_n$ and  $ \bar{v}_n$ in $ \overline{\mathbb{H}}_{n}$. Thus
we can  consider the weak symplectic form $\omega_{n}$ associated to $g_{n}$ as
above.  Therefore the maximal open set  on which $\omega_n$ is  a strong symplectic form is  the open set
$$\mathcal{U}_n=T\overline{\mathbb{H}}_n\setminus\displaystyle\bigcup_{k=1}^n(\{\frac{a}{k}\}\times\prod_{k=2}^n\mathbb{H}_k)\times \overline{\mathbb{H}}_n$$

By construction, for all $j\geq n$ and $n\in \mathbb{N}^\ast$,  we  have
$$\delta_n^j\circ{(\ell_j)_{\delta_n^j(\bar{x}_j)}}=(\ell_n)_{\bar{x}_n}\circ \delta_n^j.$$

We set $\overline{\mathbb{H}}=\underleftarrow{\lim}\overline{\mathbb{H}}_{n}$. From all the above considerations, it follows  that the sequence
$\left(  \omega_{n}\right)  _{n\in \mathbb{N}^*}$ is  a family of compatible
weak symplectic forms  which induces a weak symplectic form $\omega$ on the Fr\'echet manifold $T\overline{\mathbb{H}}\equiv\overline{\mathbb{H}}\times\overline{\mathbb{H}}$ since, as in the general case, the cotangent space
$T_{(\bar{x},\bar{u})}^{\ast}(\overline{\mathbb{H}}\times\overline{\mathbb{H}})$ does not have a Fr\'echet structure, which implies that  $\omega^\flat$ can not be surjective.\\
Now, for each $n\in \mathbb{N}^\ast$, since $(0,0)$ belongs to the open set  $\mathcal{U}_n$, we have a Darboux chart $(\overline{V}_{n},\overline{F}_{n})$ around $(0,0)\in T\overline{\mathbb{H}}_n$ from the classical Darboux Theorem for strong symplectic Banach manifold  (cf. \cite{Ma1} or
\cite{Wei} for instance). Since $\omega_n$ is  a strong symplectic form on $\overline{V}_n$ we must  have $\overline{V}_n\subset \mathcal{U}_n$.
But from the definition of $\mathcal{U}_n$,  it follows that
\[
\delta_1^n(\overline{V}_n)\cap T \overline{\mathbb{H}}_1\subset \{(x, u)\in \mathbb{H}\times\mathbb{H}\;:\; ||x||<\frac{1}{n}\}.
\]
Therefore,  according to Remark \ref{R_ProjectiveLimitTopologyFrechettoplogy},
the sequence $(\overline{V}_n, \overline{F}_n)$ is not a projective sequence of charts and so there is \textbf{no Darboux chart} for $\omega$ around $(0,0)\in T\overline{\mathbb{H}}$.
\end{example}.

\appendix\section{Projective limits of topological spaces}
\label{__ProjectiveLimitsOfTopologicalSpaces}

\begin{definition}
\label{D_ProjectiveSequenceTopologicalSpaces}
A projective sequence of topological spaces\index{projective sequence!of topological spaces} is a sequence\\
 $\left( \left(  X_{i},\delta_{i}^{j}\right) \right)_{(i,j) \in \mathbb{N}^2,\; j \geq i}$ where

\begin{description}
\item[\textbf{(PSTS 1)}]
For all $i\in\mathbb{N},$ $X_{i}$ is a topological space;

\item[\textbf{(PSTS 2)}]
For all $\left( i,j \right)\in\mathbb{N}^2$ such that $j\geq i$,
$\delta_{i}^{j}:X_{j}\to X_{i}$ is a continuous map;

\item[\textbf{(PSTS 3)}]
For all $i\in\mathbb{N}$, $\delta_{i}^{i}={Id}_{X_{i}}$;

\item[\textbf{(PSTS 4)}]
For all $\left( i,j,k \right)\in\mathbb{N}^3$ such that $k \geq j \geq i$, $\delta_{i}^{j}\circ\delta_{j}^{k}=\delta_{i}^{k}$.
\end{description}
\end{definition}

\begin{notation}
\label{N_ProjectiveSequence}
For the sake of simplicity, the projective sequence $\left( \left(  X_{i},\delta_{i}^{j}\right) \right)_{(i,j) \in \mathbb{N}^2,\; j \geq i}$ will be denoted $\left(  X_{i},\delta_{i}^{j} \right) _{j\geq i}$.
\end{notation}

An element $\left(  x_{i}\right)  _{i\in\mathbb{N}}$ of the product
${\displaystyle\prod\limits_{i\in\mathbb{N}}}X_{i}$ is called a \emph{thread}\index{thread} if, for all $j\geq i$, $\delta_{i}^{j}\left(  x_{j}\right)=x_{i}$.

\begin{definition}
\label{D_ProjectiveLimitOfASequence}
The set $X=\underleftarrow{\lim}X_{i}$\index{$X=\underleftarrow{\lim}X_{i}$} of all threads, endowed with the finest topology for which all the projections $\delta_{i}:X\to X_{i} $ are continuous, is called the projective limit of the sequence\index{projective limit!of a sequence} $\left(  X_{i},\delta_{i}^{j} \right) _{j\geq i}$.
\end{definition}

A basis\index{basis!of a topology} of the topology of $X$ is constituted by the subsets $\left( \delta_{i} \right)  ^{-1}\left(  U_{i}\right)  $ where $U_{i}$ is an open subset of $X_{i}$ (and so $\delta_i$ is open whenever $\delta_i$ is surjective).

\begin{definition}
\label{D_ProjectiveSequenceMappings}
Let $\left(  X_{i},\delta_{i}^{j} \right)  _{j\geq i}$ and $\left(  Y_{i},\gamma_{i}^{j} \right)  _{j\geq i}$ be two projective sequences whose respective projective limits are $X$ and $Y$.

A sequence $\left(  f_{i}\right)  _{i\in\mathbb{N}}$ of continuous mappings $f_{i}:X_{i}\to Y_{i}$, satisfying, for all $(i,j) \in \mathbb{N}^2,$ $j \geq i,$ the coherence condition\index{coherence condition}
\[
\gamma_{i}^{j}\circ f_{j}=f_{i}\circ\delta_{i}^{j}%
\]
is called a projective sequence of mappings\index{projective sequence!of mappings}.
\end{definition}

The projective limit of this sequence is the mapping
\[
\begin{array}
[c]{cccc}%
f: & X & \to & Y\\
& \left(  x_{i}\right)  _{i\in\mathbb{N}} & \mapsto & \left(  f_{i}\left(
x_{i}\right)  \right)  _{i\in\mathbb{N}}%
\end{array}
\]

The mapping $f$ is continuous if all the $f_{i}$ are continuous.

\section{Projective limits of Banach spaces}
\label{__ProjectiveLimitsOfBanachSpaces}
Consider a projective sequence $\left(  \mathbb{E}_{i},\delta_{i}^{j} \right)  _{j\geq i}$ of Banach spaces.
\begin{remark}
\label{R_ProjectiveSequenceOfBondingsMapsBetweenBanachSpacesDeterminedByConsecutiveRanks}
Since we have a countable sequence of Banach spaces, according to the properties of bonding maps, the sequence  $\left( \delta_i^j\right)_{(i,j)\in \mathbb{N}^2, \;j\geq i}$ is well defined by the sequence of bonding maps $\left( \delta_i^{i+1}\right) _{i\in \mathbb{N}}$.
\end{remark}
Fix some  norm $\|\;\|_i$  on $\mathbb{E}_i$, for all $i\in \mathbb{N}$. If $x=\underleftarrow{\lim}x_i$,   then  $p_n(x)=\displaystyle\max_{0\leq i\leq n} \|x_i\|_i$  is a semi-norm on the projective limit $\mathbb{F}=\underleftarrow{\lim}\mathbb{E}_n$ which provides a structure of Fr\'echet space on this vector space (see \cite{DGV}).\\
\begin{definition}
\label{D_ReducedProjectiveSequence}
A projective sequence  $\left(  \mathbb{E}_{i},\delta_{i}^{j} \right)  _{j\geq i}$ of Banach spaces is called reduced\index{reduced projective sequence}\index{projective sequence!reduced} if the range of $\delta_i^{i+1}$ is dense for all $i\in \mathbb{N}$.
\end{definition}


\begin{definition}
\label{D_EquivalentProjectiveSequencesOfBanachSpaces}
Two projective sequences  $\left(  \mathbb{E}_{i},\delta_{i}^{j} \right)  _{j\geq i}$ and  $\left(  {\mathbb{E}'}_{i},{\delta'}_{i}^{j} \right)  _{j\geq i}$ of Banach spaces are called equivalent if there exist isometries $A_i: \mathbb{E}_i\to {\mathbb{E}'}_i$ for all $i\in \mathbb{N}$ such that
\[
\delta_i^{i+1}=A_i^{-1}\circ {\delta'}_i^{i+1}\circ A_{i+1}.
\]
\end{definition}
Of course, any  projective sequence $ \left( \mathbb{E}_i,\delta_i^{j}\right)_{j\geq i}$ of Banach spaces is not  reduced and, in general, such a sequence is not equivalent to a reduced one.  However, by replacing each $\mathbb{E}_i$ by the closure $\mathbb{E}'_i$ in $\mathbb{E}_i$ of $\delta_i^{i+1}(\mathbb{E}_{i+1})$ and $\delta_i^{i+1}$ by the restriction ${\delta'}_i^{i+1}$ of $\delta_i^{i+1}$ to $ \mathbb{E}'_{i+1}$, we produce a reduced sequence of Banach spaces $\left( \mathbb{E}'_i,{\delta'}_i^{j} \right) _{j\geq i} $ such that $\underleftarrow{\lim}\mathbb{E}_i=\underleftarrow{\lim}\mathbb{E}'_i$.\\
Conversely, any Fr\'echet space provided with a countable family of semi-norms is topologically isomorphic to the projective limit of a reduced projective sequence.\\

\medskip

A particular important case of  projective limit of a reduced projective sequence of Banach spaces corresponds to the case of a decreasing sequence:
\[
\mathbb{E}_0\supset \mathbb{E}_1\supset\cdots\supset \mathbb{E}_i\supset \mathbb{E}_{i+1}\supset\cdots
\]
fulfilling, for any $i \in \mathbb{N}$, the properties:
\begin{description}
\item[\textbf{(DecS 1)}]
the inclusion $\iota_i^{i+1}:\mathbb{E}_{i+1} \to \mathbb{E}_{i}$ is continuous;
\item[\textbf{(DecS 2)}]
$\mathbb{E}_{i+1}$ is dense in $\mathbb{E}_i$.
\end{description}
Then the projective limit $\underleftarrow{\lim}\mathbb{E}_i$ is the intersection $\displaystyle \bigcap_{i\in \mathbb{N}} \mathbb{E}_i$; it is called an inverse limit of Banach spaces \index{inverse limit!of Banach spaces} or $\mathsf{ILB}$\index{$\mathsf{ILB}$} for short (cf. \cite{Omo}). In fact, any Fr\'echet space is an $\mathsf{ILB}$ space (cf. Appendix A).

\section{Projective limits of differential maps}
\label{__ProjectiveLimitsOfDifferentialMapsBetweenFrechetSpaces}
The following proposition (cf. \cite{Gal1}, Lemma 1.2) is essential
\begin{proposition}
\label{P_ProjectiveLimitsOfDifferentialMaps}
 Let $\left( \mathbb{E}_i,\delta_i^j \right) _{j\geq i}$ be a projective sequence of Banach spaces whose projective limit is the Fréchet space $\mathbb{F}=\underleftarrow{lim} \mathbb{E}_i$ and $ \left( f_i : \mathbb{E}_i \to \mathbb{E}_i  \right) _{i \in \mathbb{N}} $ a projective sequence of differential maps whose projective limit is $f=\underleftarrow{\lim} f_i$.
Then the following conditions hold:
\begin{enumerate}
\item
$f$ is smooth in the convenient sense (cf. \cite{KrMi})
\item
For all $x = \left( x_i \right) _{i \in \mathbb{N}}$, $df_x = \underleftarrow{\lim} { \left( df_i \right) }_{x_i} $.
\item
$df = \underleftarrow{\lim}df_i$.
\end{enumerate}
\end{proposition}

\section{Projective limits of Banach manifolds}
\label{__ProjectiveLimitsOfBanachManifolds}

\begin{definition}
\label{D_ProjectiveSequenceofBanachManifolds}
The projective sequence $\left( M_{i},\delta_{i}^{j} \right) _{j\geq i}$ is called \textit{projective sequence of Banach manifolds}\index{projective sequence!of Banach manifolds} if
\begin{description}
\item[\textbf{(PSBM 1)}]
$M_{i}$ is a manifold modeled on the Banach space $\mathbb{M}_{i}$;

\item[\textbf{(PSBM 2)}]
$\left(  \mathbb{M}_{i},\overline{\delta_{i}^{j}}\right) _{j\geq i}$ is a projective sequence of Banach spaces;

\item[\textbf{(PSBM 3)}]
For all $x=\left(  x_{i}\right)  \in M=\underleftarrow{\lim}M_{i}$, there exists a projective sequence of local
charts $\left(  U_{i},\varphi_{i}\right)  _{i\in\mathbb{N}}$ such that
$x_{i}\in U_{i}$ where one has the relation
\[
\varphi_{i}\circ\delta_{i}^{j}=\overline{\delta_{i}^{j}}\circ\varphi_{j};
\]

\item[\textbf{(PSBM 4)}]
 $U=\underleftarrow{\lim}U_{i}$ is a non empty open set in $M$.
\end{description}
\end{definition}

Under the assumptions   \textbf{(PSBM 1)} and  \textbf{(PSBM 2)} in Definition \ref{D_ProjectiveSequenceofBanachManifolds}, the assumptions \textbf{(PSBM 3)}] and \textbf{(PSBM 4)}  around $x\in M$ is called \emph{the projective limit chart property} around $x\in M$ and  $(U=\underleftarrow{\lim}U_{i}, \phi=\underleftarrow{\lim}\phi_{i})$ is called a \emph{projective limit chart}.

The projective limit $M=\underleftarrow{\lim}M_{i}$ has a structure of Fr\'{e}chet manifold modeled on the Fr\'{e}chet space $\mathbb{M}
=\underleftarrow{\lim}\mathbb{M}_{i}$ and is called a \emph{$\mathsf{PLB}$-manifold}\index{$\mathsf{PLB}$-manifold}. The differentiable structure is defined \textit{via} the charts $\left(  U,\varphi\right)  $ where $\varphi
=\underleftarrow{\lim}\varphi_{i}:U\to\left(  \varphi_{i}\left(U_{i}\right)  \right) _{i \in \mathbb{N}}.$\\
$\varphi$ is a homeomorphism (projective limit of homeomorphisms) and the charts changings $\left(  \psi\circ
\varphi^{-1}\right)  _{|\varphi\left(  U\right)  }=\underleftarrow{\lim
}\left(  \left(  \psi_{i}\circ\left(  \varphi_{i}\right)  ^{-1}\right)
_{|\varphi_{i}\left(  U_{i}\right)  }\right)  $ between open sets of
Fr\'{e}chet spaces are smooth in the sense of convenient spaces.

\begin{remark}
\label{R_ProjectiveLimitTopologyFrechettoplogy}
If $M$ is the projective  limit of the sequence $\left( M_{i},\delta_{i}^{j} \right) _{j\geq i}$, then, as a set, $M$ can identified  with
\[
 \left\lbrace (x_i)_{i\in\mathbb{N}}\in\prod_{i\in \mathbb{N}} M_i:\; \forall j\geq i, \; x_i=\delta_i^j(x_j)\right\rbrace.
\]
Since each $M_i$ is a topological space, we can provide $\displaystyle\prod_{i\in \mathbb{N}} M_i$ with the product topology and so, since each $\delta_i^j$ is continuous, it follows that $M$ is a closed subset in $\displaystyle\prod_{i\in \mathbb{N}} M_i$ which can be provided with the induced topology generated by the open sets of type $\displaystyle\prod_{i\in \mathbb{N}}V_i\bigcap M$ where $V_i$ is an open set of $M_i$ for a finite number of indices $i$ and otherwise $V_i=M_i$.
\end{remark}

The sequence $\left(  M_{i},\delta_{i}^{j}\right)_{j\geq i}$ is called \emph{reduced projective sequence of Banach manifolds}\index{reduced projective sequence!of Banach manifolds}
if the sequence  $\left(  \mathbb{M}_{i},\overline{\delta_{i}^{j}}\right)  _{j\geq i}$ is a reduced projective sequence of Banach spaces. Then $\delta_i^{j}(M_{j})$ is dense in $M_i$ for all $j\geq i$. We will say that   $\left(  M_{i},\delta_{i}^{j}\right)  _{j\geq i}$ is a \textit{reduced projective sequence}\index{reduced!projective sequence} and  $M=\underleftarrow{\lim}M_i$ is a  \emph{reduced $\mathsf{PLB}$-manifold}\index{reduced!$\mathsf{PLB}$-manifold}.
This situation occurs when the bonding map $\delta_i^j$ is a surjective submersion from $M_j$ onto $M_i$ for all $j\geq i$. In this case,  we say that   $\left(  M_{i},\delta_{i}^{j}\right)  _{j\geq i}$  is a \textit{surjective projective sequence} and  $M=\underleftarrow{\lim}\mathbb{M}_i$ is a  \textit{surjective $\mathsf{PLB}$-manifold}\index{surjective $\mathsf{PLB}$-manifold}\index{$\mathsf{PLB}$-manifold}.
More particular is the situation:
\begin{definition}
\label{D_SubmersiveProjectiveSequenceOf BanachManifolds}
The sequence $\left(  M_{i},\delta_{i}^{j}\right) _{j\geq i}$ is called submersive projective sequence\index{submersive!projective sequence} of Banach manifolds if
\begin{description}
\item[\textbf{(SPSBM 1)}]
$\forall (i,j) \in \mathbb{N}^2: j \geq i, \; \delta_i^j:M_j\to M_i$ is a surjective submersion;
\item[\textbf{(SPSBM 2)}]
Around each $x\in M=\underleftarrow{\lim}\mathbb{M}_i$, there exists a projective limit chart $\left( U=\underleftarrow{\lim}U_i,\varphi
=\underleftarrow{\lim}\varphi_{i} \right) $;
\item[\textbf{(SPSBM 3)}]
For all $i\in \mathbb{N}$, there exists a decomposition $\mathbb{M}_i=\ker\bar{\delta}_i^{i+1}\oplus \mathbb{M}'_i$ such that the following diagram is commutative:
\begin{eqnarray}
\label{eq_DiagramStrongProjectiveChartLimit}
\xymatrix{
U_{i+1}\ar[r]^{\varphi_{i+1}{}\;\;\;\;\;\;\;\;\;}\ar[d]_{\delta_i^{i+1}} & (\ker\bar{\delta}_i^{i+1}\times \mathbb{M}'_i)\ar[d]^{\bar{\delta}_i^{i+1}}\\
U_i\ar[r]^{\varphi_i}&\mathbb{M}_i\\
}
\end{eqnarray}
\end{description}
Such a  chart is called a submersive projective limit chart around $x$.
\end{definition}
The projective limit $M=\underleftarrow{\lim}\mathbb{M}_i$  of a submersive projective sequence  $\left( M_{i},\delta_{i}^{j}\right) _{j\geq i}$ is called a\emph{submersive projective limit of Banach manifolds} or for short a \emph{submersive $\mathsf{PLB}$-manifold}\index{submersive!$\mathsf{PLB}$-manifold}. In this case, we have the following results (cf. \cite{BCP})

\begin{proposition}
\label{P_StrongProjectiveLimitOfBanachManifolds}
Let  $\left(  M_{i},\delta_{i}^{j}\right)  _{j\geq i}$  be  a surjective (resp. submersive)  projective sequence. Then, for each $i\in \mathbb{N}$, the map $\delta_i:M\to M_i$ is surjective (resp. is a submersion).
\end{proposition}

Under the assumptions of Proposition \ref{P_StrongProjectiveLimitOfBanachManifolds},  in fact each $\delta_i^j:M_j\to M_i$ is a surjective submersion for all $j\geq i$ where $(i,j) \in \mathbb{N}^2$.\\

Another important  situation of reduced $\mathsf{PLB}$-manifold, is the case of \textit{$\mathsf{ILB}$-manifold}\index{$\mathsf{ILB}$-manifold}\index{manifold!$\mathsf{ILB}$} defined as follows:
\begin{definition}
\label{D_ILBManifold}
A $\mathsf{PLB}$-manifold $M=\underleftarrow{\lim}\mathbb{M}_i$ is called $\mathsf{ILB}$-manifold\index{$\mathsf{ILB}$-manifold} if
\begin{description}
\item[\textbf{(ILBM 1)}]
$\forall i \in \mathbb{N}, \;M_{i+1}\subset M_i$;
\item[\textbf{(ILBM 2)}]
$\forall i \in \mathbb{N}, \;\delta_i^{i+1}:M_{i+1}\to M_i$ is the canonical inclusion which is a weak immersion with dense range.
\end{description}
\end{definition}
Note that this definition is stronger than the definition of $\mathsf{ILB}$-manifold in the Omori's sense (see \cite{Omo}) since we impose the condition $\textbf{(PSBM4)}$.  In this case, $M=\displaystyle\bigcap_{i\in \mathbb{N}} M_i$.

\section{Projective limits of Banach vector bundles}
\label{__ProjectiveLimitsOfBanachVectorBundles}

Let $\left(  M_{i},\delta_{i}^{j}\right)  _{j\geq i}$ be a projective sequence of Banach manifolds where each
manifold $M_{i}$ is modeled on the Banach space $\mathbb{M}_{i}$.\\
For any integer $i$, let $\left(  E_{i},\pi_{i},M_{i}\right)  $ be the Banach
vector bundle whose type fibre is the Banach vector space $\mathbb{E}_{i}$
where $\left(  \mathbb{E}_{i},\lambda_{i}^{j}\right)  _{j\geq i}$ is a projective sequence of Banach spaces.

\begin{definition}
\label{D_ProjectiveSequenceBanachVectorBundles}
$\left( (E_i,\pi_i,M_i),\left(f_i^j,\delta_i^j \right) \right)_{j \geq i}$, where $f_{i}^{j}:E_j \to E_i$ is a morphism of vector bundles, is called a projective sequence of Banach vector bundles\index{projective sequence!of Banach vector bundles} on the projective sequence of manifolds $\left(  M_{i},\delta_{i}^{j}\right)  _{j\geq i}$ if for
all $\left(  x_{i}\right)  $ there exists a projective sequence of
trivializations $\left(  U_{i},\tau_{i}\right)  $ of $\left(  E_{i},\pi
_{i},M_{i}\right)  ,$ where $\tau_{i}:\left(  \pi_{i}\right)  ^{-1}\left(
U_{i}\right)  \to U_{i}\times\mathbb{E}_{i}$ are local
diffeomorphisms, such that $x_{i}\in U_{i}$ (open in $M_{i}$) and where
$U=\underleftarrow{\lim}U_{i}$ is a non empty open set in $M$
 where, for all $(i,j) \in \mathbb{N}^2$ such that $j\geq i,$ we have the compatibility condition
\begin{description}
\item[(\textbf{PLBVB})]
$\left(  \delta_{i}^{j}\times\lambda_{i}^{j}\right)  \circ\tau_{j}=\tau_{i}\circ f_{i}^{j}$.
\end{description}
\end{definition}

With the previous notations,  $(U=\underleftarrow{\lim}U_{i}, \tau=\underleftarrow{\lim}\tau_i)$   is called a \emph{ projective bundle chart limit}\index{projective bundle chart limit}. The triple of  projective limit
 $(E=\underleftarrow{\lim}E_{i}, \pi=\underleftarrow{\lim}\pi_{i}, M=\underleftarrow{\lim}M_{i}))$ is called a \emph{projective limit of Banach bundles} or $\mathsf{PLB}$-bundle\index{$\mathsf{PLB}$-bundle} for short. \\

The following proposition generalizes the result of \cite{Gal3} about the projective limit of tangent bundles to Banach manifolds. 

\begin{proposition}
\label{P_ProjectiveLimitOfBanachVectorBundles}
Let $\left( (E_i,\pi_i,M_i),\left(f_i^j,\delta_i^j \right) \right)_{j \geq i}$ be a projective sequence of Banach vector bundles. \\
Then $\left(  \underleftarrow{\lim}E_i,\underleftarrow{\lim}\pi_i,\underleftarrow{\lim}M_i \right)  $ is a Fr\'{e}chet vector bundle.
\end{proposition}

\begin{notation}
From now on and for the sake of simplicity, the projective sequence of vector bundles $\left( (E_i,\pi_i,M_i),\left(f_i^j,\delta_i^j \right) \right)_{j \geq i}$  will be denoted $\left(  E_{i},\pi_{i},M_{i}\right)_{\underleftarrow{i}}$.
\end{notation}

Remark that $\operatorname{GL}\left(  \mathbb{E}\right)  $ cannot be endowed with a structure
of Lie group. So it cannot play the role of structural group. We then
consider, as in \cite{Gal2}, the generalized Lie group $H^{0}\left(
\mathbb{E}\right)  =\underleftarrow{\lim}H_{i}^{0}\left(  \mathbb{E}\right)
$ which is the projective limit of the Banach-Lie groups
\[
H_{i}^{0}\left(  \mathbb{E}\right)  =\left\{  \left(  h_{1},\dots
,h_{i}\right)  \in{\displaystyle\prod\limits_{j=1}^{i}}\operatorname{GL}\left(
\mathbb{E}_{j}\right)  :\lambda_{k}^{j}\circ h_{j}=h_{k}\circ\lambda_{k}%
^{j},\text{ for }k\leq j\leq i\right\}.
\]
We then obtain the differentiability of the transition functions $\mathtt{T}$.
\begin{example}
\label{Ex_ProjectiveLimitOfBanachVectorBundles}
As a particular case of Proposition \ref{P_ProjectiveLimitOfBanachVectorBundles},   we can consider
 the projective sequence of tangent bundles $\left( (E_i,\pi_i,M_i),\left(T\delta_i^j,\delta_i^j \right) \right)_{j \geq i}$ of a projective sequence of  Banach manifolds $(M_i,\delta_i^j)_{j \geq i}$. Thus, if each $M_i$ is modeled on the Banach space $\mathbb{M}_i$,  $\left(  \underleftarrow{\lim}TM_{i},\underleftarrow{\lim}\pi
_{i},\underleftarrow{\lim}M_{i} \right)  $ is a Fr\'echet vector bundle whose typical fibre is $\mathbb{M}=\underleftarrow{\lim}\mathbb{M}_{i}$ with structural group $H^{0}\left( \mathbb{M} \right)$. As we have already seen, this result was firstly proved in \cite{Gal3}.
\end{example}

As in Appendix  \ref{__ProjectiveLimitsOfBanachManifolds}, we introduce
\begin{definition}
\label{D_StrongProjectiveLimitOfBanachBundle}  A sequence $\left(  E_{i},\pi_{i},M_{i}\right)
_{\underleftarrow{i}}$ is called a submersive  projective sequence of Banach vector bundles if $\left(E_i,\pi, M_i\right)_{ \underleftarrow{i}}$ is a submersive projective sequence of Banach manifolds and if around each $x\in M$, there exists a  projective limit chart bundle $(U=\underleftarrow{\lim}U_{i}, \tau=\underleftarrow{\lim}\tau_i)$ such that for all $i\in \mathbb{N}$, we have a decomposition $\mathbb{E}_{i+1}=\ker\bar{\lambda}_i^{i+1}\oplus \mathbb{E}'_i$ such that the condition \emph{(\textbf{PLBVB})} is true.
\end{definition}

The projective limit  $(E,\pi, M)$ of a   projective sequence of Banach vector bundles $\left(E_i,\pi, M_i\right)_{ \underleftarrow{i}}$ is called a \emph{submersive projective  limit of Banach bundles} or \emph{submersive $\mathsf{PLB}$-bundle}\index{submersive $\mathsf{PLB}$-bundle} for short.\\

Now, we have the following result whose proof is similar to Proposition \ref{P_StrongProjectiveLimitOfBanachManifolds}:

\begin{proposition}
\label{P_StrongProjectiveLimitOfBanachBundle}
Let $\left(  E_{i},\pi_{i},M_{i}\right)_{\underleftarrow{i}}$ be a submersive projective sequence of Banach bundles. Then, for each $i\in \mathbb{N}$, the  map $\lambda_i: E\to E_i$ is a submersion.
\end{proposition}

\section{The Banach space $\mathcal{H}_b\left(  \mathbb{F}_{1},\mathbb{F}_{2}\right) $}
\label{__TheFrechetSpaceHF1F2-TheBanachSpaceHbF1F2}

Let $(\mathbb{F}_{1},\nu^1_n)$ (resp. $(\mathbb{F}_{2},\nu^2_n)$)$\mathbb{\ }$ be a graded Fr\'{e}chet
space.

Recall that a linear map $L:\mathbb{F}_{1}\to \mathbb{F}_{2}$ is \emph{continuous}\index{continuous linear map}\index{linear map!continuous} if
\[
\forall n\in\mathbb{N},\exists k_{n}\in\mathbb{N},\exists C_{n}>0:\forall
x\in\mathbb{F}_{1},\nu_{2}^{n}\left(  L.x\right)  \leq C_{n}\nu_{1}^{k_{n}%
}\left(  x\right).
\]
The space $\mathcal{L}\left(  \mathbb{F}_{1},\mathbb{F}_{2}\right)  $ of
continuous linear maps between both these Fr\'{e}chet spaces generally drops
out of the Fr\'{e}chet category. Indeed, $\mathcal{L}\left(  \mathbb{F}%
_{1},\mathbb{F}_{2}\right)  $ is a Hausdorff locally convex topological vector
space whose topology is defined by the family of semi-norms $\left\{
p_{n,B}\right\}  $:%
\[
p_{n,B}\left(  L\right)  =\displaystyle\sup_{x\in
	B}\left\{  \nu^{2}_{n}\left(  L.x\right) \right\}
\]
where $n\in\mathbb{N}$ and $B$ is any bounded subset of $\mathbb{F}_1$. This topology is not metrizable since the
family $\left\{  p_{n,B}\right\}  $ is not countable.\\
So $\mathcal{L}\left(  \mathbb{F}_{1},\mathbb{F}_{2}\right)  $ will be replaced,
under certain assumptions, by a projective limit of appropriate functional
spaces as introduced in \cite{Gal2}.

We denote by $\mathcal{L}\left(  \mathbb{B}_{1}^{n},\mathbb{B}_{2}^{n}\right) $ the space of linear continuous maps (or equivalently bounded
linear maps because $\mathbb{B}_{1}^{n}$ and $\mathbb{B}_{2}^{n}$ are normed
spaces). We then have the following result (\cite{DGV}, Theorem 2.3.10).

\begin{theorem}
\label{T_HF1F2}
The space of all continuous linear maps between $\mathbb{F}_{1}$ and $\mathbb{F}_{2}$ which can be represented as projective limits
\[
	\mathcal{H}\left(  \mathbb{F}_{1},\mathbb{F}_{2}\right)  =\left\{  \left(
	L_{n}\right)  \in\prod\limits_{n\in\mathbb{N}}\mathcal{L}\left(
	\mathbb{B}_{1}^{n},\mathbb{B}_{2}^{n}\right)  :\underleftarrow{\lim}%
	L_{n}\text{ exists}\right\}
\]
\index{$\mathcal{H}\left(  \mathbb{F}_{1},\mathbb{F}_{2}\right)$}is a Fr\'{e}chet space.
\end{theorem}

For this sequence $\left(  L_{n}\right) _{n \in \mathbb{N}}  $ of linear maps, for any integer
$0\leq n\leq m$, the following diagram is commutative
\[
\xymatrix{
           \mathbb{B}_{1}^{n}  \ar@{<-}[r]^{\;\;\;{(\delta_1)}_{n}^{m}} \ar[d]_{L_n} & \mathbb{B}_{1}^{m} \ar[d]^{L_m}\\
           \mathbb{B}_{2}^{n} \ar@{<-}[r]^{\;\;\;{(\delta_2)}_{n}^{m}}               &  \mathbb{B}_{2}^{m}\\
}
\]
On $\mathcal{H}\left(  \mathbb{F}_{1},\mathbb{F}_{2}\right) $, the topology can be defined by the sequence of seminorms $p_n$ given by
\[
p_n\left(  L\right)  =\displaystyle\max_{0\leq k\leq n} \sup\left\{  \nu^{2}_{k}\left(  L.x\right)  ,x\in \mathbb{F}_1,\; \nu^1_k(x)=1
\right\}
\]
so that $\left(\mathcal{H}\left(  \mathbb{F}_{1},\mathbb{F}_{2}\right) ,p_n\right)$ is a graded Fr\'echet space.

\begin{remark}
	\label{R_SeminormEquivalentpn}
	 For $l\in\left\lbrace 1,2 \right\rbrace$ , given a graduation $ \left( \nu^l_n \right)$  on a Fr\'echet space $\mathbb{F}_l$, let $\mathbb{B}_l^n$ be  the associated local Banach space and $\delta_l^n:\mathbb{F}_l\to \mathbb{B}_l^n$ the canonical projection.\\
	 The quotient norm  $\tilde{\nu}^l_n$  associated to $\nu^l_n$ is defined by
	\begin{eqnarray}
	\label{eq_DefTildenu}
	\tilde{\nu}^l_n(\delta_n(z))=\sup\{\nu^l_n(y):\; \delta_n(y)=\delta_n(z) \}.
	\end{eqnarray}
	We denote by $(\tilde{\nu}^2_n)^{\operatorname{op}}$ the corresponding operator norm on $\mathcal{L}(\mathbb{B}_1^n,\mathbb{B}_2^n)$.\\
	If $L= \underleftarrow{\lim}L_n$ where $L_n:\mathbb{B}_1^n\to\mathbb{B}_2^n$, then we have
	\[
	(\tilde{\nu}^2_n)^{\operatorname{op}}(L_n)=\sup\{\tilde{\nu}^2_n(L_n.x),\;\;x\in \mathbb{B}_1^n\;\; \tilde{\nu}^1_n(x)\leq 1\}=\sup\{\nu^2_n(L.x), x\in \mathbb{F}_1, \nu^1(x)\leq 1\}.
	\]
	This implies that
	\[
	p_n(L)=\displaystyle\max_{0\leq i\leq n}(\tilde{\nu}^2_i)^{\operatorname{op}}(L_n).
	\]
\end{remark}

\begin{definition}
	\label{D_UniformlyBoundedOperator}
	Let $(\mathbb{F}_{1},\nu^1_n)$ and  $(\mathbb{F}_{2},\nu^2_n)$ be  graded Fr\'{e}chet spaces.
	A linear map $L:\mathbb{F}_{1}\to \mathbb{F}_{2}$ is called a uniformly bounded operator\index{uniformly bounded operator}\index{operator!uniformly bounded}, if
	\[
	\exists C>O : \forall n\in \mathbb{N}, \; \nu_n(L(x))\leq C\mu_n(x).
	\] 	
\end{definition}

We denote by $\mathcal{H}_b\left(  \mathbb{F}_{1},\mathbb{F}_{2}\right) $\index{$\mathcal{H}_b\left(  \mathbb{F}_{1},\mathbb{F}_{2}\right) $} the set of uniformly bounded operators. Of course $\mathcal{H}_b\left(  \mathbb{F}_{1},\mathbb{F}_{2}\right) $ is contained in $\mathcal{H}\left(  \mathbb{F}_{1},\mathbb{F}_{2}\right)$ and $L\in\mathcal{H}\left(  \mathbb{F}_{1},\mathbb{F}_{2}\right)$  belongs to $\mathcal{H}_b\left(  \mathbb{F}_{1},\mathbb{F}_{2}\right) $ if and only if $\displaystyle\sup_{n\in \mathbb{N}} p_n(L)<\infty$ and so
\[
\mathcal{H}_b\left(  \mathbb{F}_{1},\mathbb{F}_{2}\right) =\left[\mathcal{H}\left(  \mathbb{F}_{1},\mathbb{F}_{2}\right)\right]_b.
\]
When $\mathbb{F}=\mathbb{F}_1=\mathbb{F}_2$ and $\nu^1_n=\nu^2_n$ for all $n\in \mathbb{N}$, the set $\mathcal{H}\left(  \mathbb{F},\mathbb{F}\right)$ (resp. $\mathcal{H}_b\left(  \mathbb{F},\mathbb{F}\right))$ is simply denoted $\mathcal{H}\left(  \mathbb{F}\right)$ (resp.  $\mathcal{H}_b\left(  \mathbb{F}\right)$). \\

We denote by $\mathcal{IH}_{b}\left(  \mathbb{F}_{1},\mathbb{F}_{2}\right) $\index{$\mathcal{IH}_{b}\left(  \mathbb{F}_{1},\mathbb{F}_{2}\right) $}  (resp. $\mathcal{SH}_{b}\left(  \mathbb{F}_{1},\mathbb{F}_{2}\right) $\index{$\mathcal{SH}_{b}\left(  \mathbb{F}_{1},\mathbb{F}_{2}\right) $}) the set of injective  (resp. surjective)  operators of $\mathcal{H}_b\left(  \mathbb{F}_{1},\mathbb{F}_{2}\right) $ with closed range.
\begin{proposition} (\cite{BCP})
\label{P_InjectiveSurjectiveBH}
${}$
\begin{enumerate}
		\item
 Each operator $L\in \mathcal{H}\left(  \mathbb{F}_{1},\mathbb{F}_{2}\right) $ has a closed range if and only if,  for each $n\in \mathbb{N}$, the induced operator $L_n:\mathbb{B}_1^n\to\mathbb{B}_2^n$ has  a closed range.
		\item
 $\mathcal{IH}_{b}\left(  \mathbb{F}_{1},\mathbb{F}_{2}\right) $ is an open subset of  $\mathcal{H}_b\left(  \mathbb{F}_{1},\mathbb{F}_{2}\right) $.
		\item
 $\mathcal{SH}_{b}\left(  \mathbb{F}_{1},\mathbb{F}_{2}\right) $ is an open subset of  $\mathcal{H}_b\left(  \mathbb{F}_{1},\mathbb{F}_{2}\right) $.
\end{enumerate}
\end{proposition}

We are in situation to end this section by the following result:

\begin{theorem}(\cite{BCP})
\label{T_UniformlyBoundedProperties}
${}$
\begin{enumerate}
\item
The Banach space $\mathcal{H}_b(\mathbb{F})$ has a Banach-Lie algebra structure and the set $\mathcal{GH}_b(\mathbb{F})$ of uniformly bounded isomorphisms of $\mathbb{F}$ is open in $\mathcal{H}_{b}(  \mathbb{F})$.
\item
$\mathcal{GH}_b(\mathbb{F})$ has a structure of  Banach-Lie group whose Lie algebra is $\mathcal{H}_b(\mathbb{F})$.
\item
If $\mathbb{F}$ is identified with the projective $\underleftarrow{\lim} \mathbb{B}^n$ we denote by $\exp_n:\mathcal{L}(\mathbb{B}_n) \to \mathcal{GL}(\mathbb{B}_n)$, then we a have a well defined smooth map $\exp:=\underleftarrow{\lim}\exp_n: \mathcal{H}_b(\mathcal{F}) \to \mathcal{GH}_b(\mathbb{F})$ which is a diffeomorphism from an open set of $0 \in \mathcal{H}_b(\mathcal{F})$ onto a a neighbourhood of $\operatorname{Id}_\mathbb{F}$.
\end{enumerate}
\end{theorem}

\section{A theorem of existence of ODE}

 The following result is in fact a reformulation in our context of Theorem 1 in \cite{Lob}.  
 
 \begin{theorem}
	\label{T_SolutionODEOnFrechetSpaces}
	Let $\mathbb{F}$ a Fr\'echet space realized as the limit of a surjective projective sequence of Banach spaces
	$ \left( \mathbb{B}_n,\lambda_n^m \right)_{m \geq n}$ whose topology is defined by the sequence of seminorms $\left( \nu_n\right)_{n \in \mathbb{N}}$.  Let $I$ be an open interval in $\mathbb{R}$ and $U$ be an open set of $I\times \mathbb{F}$. Then  $U$ is a surjective projective limit of open sets $U_n \subset I \times \mathbb{B}_n$. Consider a smooth map $f=\underleftarrow{\lim}f_n: U \to \mathbb{F}$, projective limit of maps $f_n: U_n\to \mathbb{B}_n$.
\footnote{This means that we have: $ \forall m\geq n, \; \lambda_n^m\circ f_m=f_n\circ (Id_\mathbb{R}\times \lambda_n^m)$}
	Assume that  for every point $\left( t,x) \right)\in U$, and every $n\in \mathbb{N} $, there exists an integrable function $K_n>0$ such that
	\begin{equation}
		\label{eq_LipschitzAssumption}
		\forall \left( (t,x), (t,x')\right) \in U^2, \; \nu_n(f(t,x)-f(t,x'))\leq K_n(t)\nu_n(x-x').
	\end{equation}
	and consider the differential equation:
	\begin{equation}
		\label{eq_ODEInFrechetSpace}
		\dot{x}=\phi \left( t,x \right).
	\end{equation}
	\begin{enumerate}
		\item
For any $(t_0,x_0)\in U$, there exists  $\alpha >0$  with $I_\alpha=[t_0-\alpha, t_0+\alpha]\subset I$, an open pseudo-ball $V=B(x_0,r)\subset U$ and a map $\Phi: I_\alpha\times I_\alpha \times V\to \mathbb{F}$ such that
		\[
		t\mapsto \Phi(t,\tau, x)
		\]
		is the unique solution of (\ref{eq_ODEInFrechetSpace})
		with initial condition $\Phi(\tau,\tau, x )=x$ for all $x\in V$.
		\item
$V$ is the  projective limit of the open  balls  $V_n$ of $\mathbb{B}_n$.  For each $n\in \mathbb{N}$, the curve $t\mapsto \lambda_n\circ \Phi(t,\tau, \lambda_n(x))$ is the unique solution $\gamma:I_\alpha\to \mathbb{B}_n$ of the differential equation $\dot{x}_n=\phi_n \left( t,x_n \right)$ with initial condition $\gamma(\tau)=\lambda_n(x)$.
	\end{enumerate}
\end{theorem}

\section{Proof of Proposition \ref{P_PropertiesWeakIsometricLinear2Forms}}\label{__ProofProposition12}

The proof of this Proposition \ref{P_PropertiesWeakIsometricLinear2Forms} needs the following Lemma :

\begin{lemma}
\label{L_EProductOfBanachSpaces}
Let $\mathbb{E}$ be a projective limit of a reductive projective sequence $\left( \mathbb{E}_i, \ell_i^j \right) _{j\geq i}$. Assume that, for all $(i,j)\in \mathbb{N}^2$ such that $j\geq i$, the kernel of $\ell_i^j$ is supplemented.
\begin{enumerate}
\item
For each $i\in \mathbb{N}$ and each $j\geq i$ we have a decomposition
\begin{eqnarray}\label{eq_DecompositionEl}
\mathbb{E}_j=\mathbb{E}_{i}^j\oplus\mathbb{E}_{i+1}^j\oplus\cdots \oplus\mathbb{E}_{j-1}^j\oplus \ker\ell_{j-1}^{j}
\end{eqnarray}
with the following properties for all $j \geq i$
\begin{enumerate}
\item[(a)]
$\ker\ell_l^j=\mathbb{E}_{l+1}^j\oplus\cdots \oplus\mathbb{E}_{j-2}^j\oplus \ker\ell_{j-1}^{j}$;
\item[(b)]
the restriction of  $(\ell')_l^j$ of $\ell_l^j$ to $(\mathbb{E}_{i}^j\oplus\cdots\oplus \mathbb{E}_{l}^j)$ is  injective with dense range in $\mathbb{E}_l$;
\item[(c)]
$\ell_l^j( \mathbb{E}_h^j)$ is dense in $ \mathbb{E}_h^l $  for all $ i\leq h\leq l$, is dense in $\ker\ell_{l-1}^l$  for $h=l$ and  $\ell_l^j( \mathbb{E}_h^j)=\{0\}$ for $l< h\leq j-1$.
\end{enumerate}
\item
Let $\ell_i:\mathbb{E}\to \mathbb{E}_i$ the canonical projection. Then $E=\ker\ell_i\oplus \mathbb{F}'_i$ and the restriction $\ell_i'$ of $\ell$ to $\mathbb{E}'_i$ is a continuous map injective map into $\mathbb{E}_i$ with dense range.  Moreover, if $||\;||_i$ is a norm on $\mathbb{E}_i$, then $\nu_i=||\;||_i\circ \ell_i$ is a semi-norm on $E$ and the restriction of  $\nu_i$ to $\mathbb{F}'_i$ is a norm and in this case, $\ell'_i$ is an isometry. In particular, the completion $\overline{\mathbb{F}'_i}$ of $\mathbb{F}'_i$ is isomorphic to $\mathbb{E}_i$.
\item
We set $\mathbb{K}_i=\ker\ell_{i-1}^i$ for $i \geq 1$ and $\mathbb{K}_0=\mathbb{E}_0$.  If $\mathbb{E}'_j=\displaystyle\prod_{l=0}^j\mathbb{K}_l$, then there exist bounding maps  $\kappa_i^j:\mathbb{E}'_j\to \mathbb{E}'_i$ with dense range, so that $\left( \mathbb{E}'_i,\kappa_i^j \right)_{j\geq i}$ is a reduced projective sequence. If $\mathbb{E}'=\underleftarrow{\lim}\mathbb{E}'_i$ there exists an injective continuous linear map $\underleftarrow{\lim}\theta_i:\mathbb{E}\to\mathbb{E}'$ with dense range where $\theta_i$ is an injective linear map from $\mathbb{E}_i$ into $\mathbb{E}'_i$ with dense range. Moreover, if $\left( \mathbb{E}_i, \ell_i^j \right) _{j \geq i}$ if a surjective projective sequence, then each $\theta_i$, $i\in \mathbb{N}$  and $\theta$ are isomorphisms.
\end{enumerate}
\end{lemma}

\begin{remark}
\label{R_ILB-case}
Note that, if $\left( \mathbb{E}_i, \ell_i^j \right) _{j \geq i}$ is a $\mathsf{ILB}$ sequence\footnote{ cf. Definition \ref{D_ILBManifold}}, the decomposition \ref{eq_DecompositionEl} is reduced to $\mathbb{E}_i^j=\mathbb{E}_i\cap \mathbb{E}_j=(\ell_i^j )^{-1}(\mathbb{E}_i )$. We have $\mathbb{E}'_i=\mathbb{E}_0$ for all $i\in\mathbb{N}$ and $\theta_i$ is the inclusion of $\mathbb{E}_i$ in $ \mathbb{E}_0$. Thus, in this case, the morphism $\theta:\mathbb{E} \to \mathbb{E}'$ is simply the injection of $\mathbb{E}=\displaystyle\bigcap_{i\in \mathbb{N}}\mathbb{E}_i$ into $\mathbb{E}_0$. This means that the only interesting context of Lemma \ref{L_EProductOfBanachSpaces} is when $\ker\ell_i^l$ is not reduced to zero for some pairs $(j,i)$ and $j\geq i$.
\end{remark}

\begin{proof}
Fix some $i\in \mathbb{N}$ and assume  that,  for all  $i\leq l \leq j$, we have a decomposition of type (\ref{eq_DecompositionEl}) with properties (b) and (c).\\
It is clear that this assumption is true for $l=i$ and $l=i+1$. At first,
 we have a decomposition
 \[
 \mathbb{E}_{j+1}=\ker\ell_j^{j+1}\oplus \mathbb{F}_{j+1}.
 \]
 Therefore the restriction $(\ell')_j^{j+1}$ of $\ell_j^{j+1}$ to $\mathbb{F}_{j+1}$ is an injective continuous map from $\mathbb{F}_{j+1}$ into $\mathbb{E}_j$ and  $(\ell')_j^{j+1}$ and $\ell_i^{j+1}$ have the same range, so $(\ell')_j^{j+1}(\mathbb{F}_{j+1})$ is dense in $\mathbb{E}_j$.
 Therefore,  according to (\ref{eq_DecompositionEl}), each vector space $\mathbb{K}_j^{j+1}=[(\ell')_j^{j+1}]^{-1}(\ker\ell_{j-1}^j)\;$, $\mathbb{K}_l^{j+1}= [(\ell')_j^{j+1})]^{-1}(\mathbb{K}_l^j)\;$ for all $i\leq l<j$
   are Banach subspaces of $\mathbb{F}_{j+1}$ and we have the following decomposition:
 \[
 \mathbb{F}_{j+1}=\mathbb{E}_i^{j+1}\oplus\mathbb{E}_{i+1}^{j+1}\oplus\mathbb{E}_{i+2}^{j+1}\oplus\cdots \oplus\mathbb{E}_{j}^{j+1}.
 \]
It follows that,  for $j > l > i$, we have
\begin{align*}
(\ell_j^{j+1})^{-1}(\ker \ell_l^j)  & =(\ell_j^{j+1})^{-1}(\mathbb{E}_{l+1}^j\oplus\cdots \oplus\mathbb{E}_{j-1}^j\oplus \ker\ell_{j-1}^j) \\
                                    & =\mathbb{E}_{l+1}^{j+1}\oplus\cdots \oplus\mathbb{E}_{j-1}^{j+1}\oplus \mathbb{E}_{j-1}^j\oplus \ker\ell_j^{j+1}
\end{align*}
and also:
\begin{align*}
\ell_l^{j+1}(\mathbb{E}_i^{j+1}\oplus\mathbb{E}_{i+1}^{j+1}\oplus\cdots\oplus \mathbb{E}_{j}^{j+1})  & =\ell_l^j(\ell_j^{j+1}(\mathbb{E}_i^{j+1}\oplus\mathbb{E}_{i+1}^{j+1}\oplus\cdots\oplus \mathbb{E}_{j}^{j+1}) \\
& =\ell_l^j(\mathbb{E}_i^{j}\oplus\mathbb{E}_{i+1}^{j}\oplus\cdots\oplus \mathbb{E}_{j-1}^{j}\oplus \ker\ell_{l-1}^l)\\
& =\ell_l^j(\mathbb{E}_j)
\end{align*}
which is dense in $\mathbb{E}_l$.\\
For $l=j$, we have
\[
\ell_j^{j+1}(\mathbb{E}_i^{j+1}\oplus\mathbb{E}_{i+1}^{j+1}\oplus\mathbb{E}_{i+2}^{j+1}\oplus\cdots \oplus\mathbb{E}_{j}^{j+1})
=\ell_j^{j+1}(\mathbb{E}_{j+1})
\]
which is dense in $\mathbb{E}_j$.
Finally, according to the definition of  $(\ell')_j^{j+1}$ and the definition $\mathbb{E}_h^{j+1}$ it follows that   $(\ell')_j^{j+1}$ is an injective continuous map from  $\mathbb{E}_h^{j+1}$ into  $\mathbb{E}_h^{j}$  with dense range for all $i\leq h<j+1$.
Thus (1) is proved.\\

For $i$ fixed, on the one hand,  the sequence $\left( \ker \ell_i^l,( \ell_l^j)_{|\ker\ell_i^j} \right) _{j\geq i}$ is a projective system of Banach spaces according to properties  (1) and (3) in the first part of Lemma \ref{L_EProductOfBanachSpaces}. Since $\ell_i=\underleftarrow{\lim}\ell_i^j$, then $\ker\ell_i^j=\underleftarrow{\lim}\ker\ell_i^j$ which is a closed Fr\'echet subspace of $\mathbb{E}$.

On the other hand, we have $\mathbb{E}_j=\ker\ell_i^j\oplus \mathbb{E}_i^j$ and from (3), the sequence $\left( \mathbb{E}_i^l, (\ell_l^j)_{| \mathbb{E}_i^j} \right) _{l\geq j}$ is a projective sequence of Banach spaces. Thus if $\mathbb{F}'_i=\underleftarrow{\lim}\mathbb{E}_i^j$, we have $E=\ker\ell_i\oplus \mathbb{F}'_i$. It follows that the restriction $\ell'_i$ of $\ell_i$ to $\mathbb{F}'_i$ is an isomorphism onto $\ell_i(\mathbb{F}'_i)$ which is dense in $\mathbb{E}_i$. If $||\;||_i$ is a norm on $\mathbb{E}_i$, then $\nu_i=||\;||\circ \ell_i$ is a semi-norm on $\mathbb{E}$ whose kernel is precisely $\ker\ell_i$. Thus, the restriction of $\nu_i$ to $\mathbb{F}'_i$ is a norm and so $\ell'_i$ is an isometry  which ends the proof of (2).

(3) We set $\mathbb{K}_l=\ker\ell_{l-1}^l$ and $\mathbb{E}'_j=\displaystyle\prod_{l=0}^j \mathbb{K}_l$. We consider the followings maps:
\begin{description}
\item
$ \theta_j: \mathbb{E}_j\to\mathbb{E}'_j$ defined in the following way for $i\in \mathbb{N}$:
according to the decomposition (\ref{eq_DecompositionEl}), each $x^j\in \mathbb{E}_j$ can be written as a sequence $ (x^j_l)_{0\leq l\leq j}$  where $x^j_l$ belongs to $\mathbb{E}^j_l$ and so  $\theta_j\left((x^j_l)\right)=\ell^j_l\left((x^j_l)\right) $  belongs to $\mathbb{E}'_i$ from property (3)
\footnote{Take $h=l$ in property (3)}
and so is well defined.
\item
$\kappa_i^j:\mathbb{E}'_j\to \mathbb{E}'_i$ defined in the following way:
if $\overline{x'}_i=(x'_0,\dots,x'_l\dots,x'_i)$ and $\overline{y'}_j=(y'_0,\dots,y'_l\dots,y'_j)$ then

 $\kappa_i^j\left(\overline{x'}_j=(x'_l)\right):=\left((\ell_l^j(x'_l)\right)$ for $0\leq l\leq i$ and for all $i\in \mathbb{N}$ and $j\geq i$.
\end{description}

Note that $\theta_j$ is a continuous linear map which is injective. Indeed, if $(x'_l)$ belongs to $\theta_j(\mathbb{E}_j)$ then $x'_l$ belongs to $\ell_l^j(\mathbb{E}_l^j)\subset \ker\ell_{l-1}^l$. The restriction $\ell_l^j$ to $\mathbb{E}_l^j$ being injective according to property (2), we have a unique $x_l^j\in \mathbb{E}_l^j$ such that $\ell_l^j(x_l^j)=x'_l$ 
for all $l \in \left\{ 0, \dots, j \right\} $. But, from property (3), $\ell_l^j(\mathbb{E}_l^j)$ is dense in $\mathbb{K}_l $, thus $\theta_j$ has a dense range in $\mathbb{E}'_j$. According to these notations, we can see that we have $\theta_i\circ \ell_i^j=\kappa_i^j\circ \theta_j$
It  follows that we get an injective continuous linear map $\theta=\underleftarrow{\lim}\theta_i$ from $\mathbb{E}=\underleftarrow{\lim}\mathbb{E}_i$ to $\mathbb{E}'=\underleftarrow{\lim}\mathbb{E}'_i$
But $\mathbb{E}$ can be identified with (cf. Appendix \ref{__ProjectiveLimitsOfTopologicalSpaces})
$$\left\{(x_i)\in \prod_{i\in \mathbb{N}}\mathbb{E}_i,\;:\; x^i_l=\ell_i^j(x^j_l),\;\ 0\leq l\leq i ,\; j\geq i,\; i\in \mathbb{N}\right\}$$
In the same way,   $\mathbb{E}'$ can be identified with
\[
\left\{(\overline{x'_i})\in \prod_{i\in \mathbb{N}}\mathbb{E}'_i,\;:\; x'_l=\kappa_i^j(x'_l),\; 0\leq l\leq i,\; j\geq i,\; i\in \mathbb{N}\right\}.
\]
This implies that $\theta$ is a continuous injective map from $\mathbb{E}$ to $\mathbb{E}'$ with dense range.

Now assume that  $ \ell_i^j$ is surjective for all $j\geq i$ and $\left( i,j \right) \in \mathbb{N}^2$. Then clearly this implies that $\kappa_i^j$ is also surjective and so $\theta_j$ is an isomorphism for all $j\geq i$. In this way, $\theta$ is also an isomorphism.

\end{proof}

\begin{proof}[Proof of Proposition \ref{P_PropertiesWeakIsometricLinear2Forms}] 
From our assumption, on  the sequence $(\omega_i)_ {i \in \mathbb{N}}$
and Proposition \ref{P_IsometricProperties}, we have a decomposition $\mathbb{E}_i=\ker\ell_{i-1}^i\oplus\mathbb{F}_i$ with  $\mathbb{F}_i=(\ker\ell_{i-1}^i)^\perp$ for all $i\geq 1$.
Thus we can apply lemma \ref{L_EProductOfBanachSpaces}. Thus,  for all $1\leq j\leq n$,  we have a  decomposition
\[
\mathbb{E}_{j}=\ker\ell_{j-1}^{j}\oplus \mathbb{F}_{j} \textrm{ with }
\mathbb{F}_{j}=\mathbb{E}_0^{j}\oplus \mathbb{E}_{1}^{j}\oplus\mathbb{E}_{2}^{j}\oplus\cdots \oplus\mathbb{E}_{j-1}^{j}.
\]

At first, $\omega_0$ is  a symplectic form  $\omega_{\mathbb{K}_0}$ on $\mathbb{K}_0=\mathbb{E}_0$. From the assumption  on  the sequence $(\omega_i)_ {i \in \mathbb{N}}$,  by induction on $i$,  $\omega_i$ induces a symplectic form $\omega_{\mathbb{K}_i}$ on $\mathbb{K}_i$. In this way, we obtain a symplectic form  $\omega'_i$ on $\mathbb{E}'_i$ by
\begin{eqnarray}\label{eq_omega'i}
\omega_{i}'(\overline{u'}_i,\overline{v'}_i)=\sum_{l=0}^i \omega_{\mathbb{K}_l}(u'_l, v'_l)
\end{eqnarray}
if $\overline{u'}_i=(u'_0,\dots u'_i)$ and $\overline{v'}_i=(v'_0,\dots v'_i)$.
According to the notations of the proof of Lemma \ref{L_EProductOfBanachSpaces}, for all $\left( i,j \right) \in \mathbb{N}^2$ such that $j\geq i$,  from the definition of $\kappa_i^j$ and $\theta_j$,  it is easy to see that
\begin{eqnarray}\label{eq_coherenceomega'}
{\omega_j}_{| \mathbb{K}_l} =\left((\kappa_i^j)^*\omega'_i\right)_{| \mathbb{K}_l}\,\;:\; \forall\; \; 0\leq l\leq i \textrm{ and }\omega_i=\theta_i^*\omega'_i.
\end{eqnarray}
Now the sequence $\left( \mathbb{E}'_i\times\mathbb{E}'_i,\kappa_i^j\times\kappa_i^j \right)  _{j\geq i}$ is a reduced projective system. Thus according to (\ref{eq_coherenceomega'}), we can define  a $2$-form $\omega'$ on $\mathbb{E}'$ in the following way:\\

if $\overline{u'}=(u'_l)\in E$ and $\overline{v'}=(v'_l)\in \mathbb{E}'$ then
$$\omega'(\overline{u'},\overline{v'})=\underleftarrow{\lim}\omega'_i\left((u'_0,\dots,u'_i),(v'_0,\dots,v'_i)\right)$$

It remains to show that $\omega'$ is symplectic. At first, note that by construction of  $\omega'_i$ we have $\omega'_i(u, v)=0$ for all $u\in \mathbb{K}_j$ and $v\in \displaystyle\prod_{l\not=j, 0\leq l\leq i}\mathbb{K}_l$. Assume $\omega'(\overline{u'},\overline{v'})=0$ for all $\overline{v'}\in \mathbb{E}'$ this implies that for all $l\in \mathbb{N}$, we have  $\omega_{\mathbb{K}_l}(u'_l, v'_l)=0$ for all $v_l\in \mathbb{K}_l$ and so we must have have $u'_l=0$ for all $l\in \mathbb{N}$.  Now, since  $\theta=\underleftarrow{\lim}\theta_i$ is injective, and $\omega_i=\theta_i^*\omega'_i$ and the range of $\theta$ is dense this implies the results for $\omega$.\\
\end{proof}

\end{document}